\documentclass[11pt]{article}
\usepackage{amsmath}
\usepackage{amsfonts}
\usepackage{amssymb}
\usepackage{amsthm }
\usepackage[backend=bibtex]{biblatex}
\usepackage{xcolor}
\usepackage{graphicx}
\usepackage{cleveref}
\usepackage{enumerate}
\usepackage{epstopdf}
\usepackage{fullpage}
\usepackage{mathtools}

\usepackage{authblk}

\newtheorem{theorem}{Theorem}[section]
\newtheorem{proposition}[theorem]{Proposition}

\newtheorem{lemma}[theorem]{Lemma}
\newtheorem{example}[theorem]{Example} %
\newtheorem{remark}[theorem]{Remark}%
\usepackage{mathtools}
 
\newtheorem{definition}[theorem]{Definition}%
\DeclarePairedDelimiter\floor{\lfloor}{\rfloor}

\newcommand{\E}{\mathbb{E}}

\newcommand{\cl}{\mathcal}
\newcommand{\bb}{\mathbb}

\newcommand{\eqd}{\overset{d}{=}}
\newcommand{\eqfdd}{\overset{fdd}{=}}
\newcommand{\wt}{\widetilde}
\author{Shuyang Bai* \\   bsy9142@uga.edu  
   \and He Tang \\  ht11145@uga.edu  }

\begin{document}

\title{Joint sum-max limit for a class of long-range dependent processes with heavy tails}

\maketitle
\def\thefootnote{*}\footnotetext{Corresponding Author. The  authors are ordered alphabetically and contributed equally to this work.}

\begin{abstract}
We consider a class of   stationary   processes exhibiting both long-range dependence and  heavy tails.  Separate limit theorems for sums and   for extremes have been established recently in literature  with novel objects appearing in the limits.  
In this article,	we establish  the joint sum-max limit theorems   for this class of processes.  In the finite-variance case, the limit consists of two independent components: a fractional Brownian motion arising from the sum, and a long-range dependent random sup measure arising from the maximum. In the infinite-variance  case, we obtain in the limit two dependent components: a stable process  and a random sup measure whose dependence structure is described through the local time and range of a stable subordinator.  For establishing the limit theorem in the latter case, we also develop a  joint   convergence result for the local time and range of subordinators, which may be of independent interest.
\end{abstract}

{\it Keywords:} Asymptotic dependence,  infinitely divisible processes, long-range dependence, stable subordinator, 
weak convergence

Mathematics Subject Classification (2020) 60F17 (1st);   60G10(2nd)

\section{Introduction}\label{sec1}
%

%
%

Let $\{X_n\}=\{X_n\}_{n\in\mathbb{N}}$, $\mathbb{N}=\{1,2,\ldots\}$, be a sequence of stationary random variables. We are interested in the asymptotic behavior of the joint distribution of  partial sum $S_n=X_1+\ldots+X_n$ and partial maximum $M_n=\max(X_1,\ldots,X_n)$. 

The case when $X_n$'s are independent identically distributed (iid) was thoroughly studied by Chow and Teugels \cite{chow1978sum}. They showed that properly normalized  $(S_n,M_n)$   converges jointly in distribution to a non-degenerate limit  $(S,M)$    as $n\rightarrow\infty$, if and only if each marginal convergence holds. Furthermore,  $S_n$ and $M_n$ are asymptotically independent, namely,  $S$ and $M$ are independent, unless $X_n$ has a heavy tail on the positive side, that is, the tail    $\mathbb{P}(X_n>x)$  is  regularly varying  with index $-\alpha$, $\alpha
\in (0,2)$ as $x\rightarrow\infty$; see \cite{bingham1989regular} for the notion of regular variation.   A physical explanation for the asymptotic independence  is that each light-tailed individual $X_n$ is asymptotically negligible in $S_n$, and hence $M_n$ should play no role in determining $S$. In contrast, when $X_n$ has a heavy tail,    the asymptotic dependence of $S_n$ and $M_n$ is due to the fact that occasional extremely large observations  dominate  the sum of all others. In this case, marginally the limit $S$ follows an $\alpha$-stable distribution, and the limit $M$ follows an $\alpha$-Fréchet distribution,  while $S$ and $M$ are dependent. Writing $S_n(t)=S_{\floor{nt}}$ and  $M_n(t)=M_{\floor{nt}}$, where $\floor {x}$  denotes the  greatest integer not exceeding $x$, Chow and Teugels \cite{chow1978sum} also extended the result to the functional convergence of properly normalized $(S_n(t) ,M_n(t))_{t\ge 0}$ in  a Skorokhod space to a  joint limit process  $(S(t) ,M(t))_{t\ge 0}$.   In the heavy-tailed case, the limit processes $S(t)$ and $M(t)$ are dependent, while marginally $S(t)$ is an $\alpha$-stable Lévy process and $M(t)$ is  an   $\alpha$-Fréchet extremal process.

What if $\{X_n\}$ has  dependence? Typically, under  a proper  weak  dependence, or say,  a \emph{short-range dependence} assumption, and
if $X_n$ has light tails, the joint distribution of normalized $(S_n,M_n)$ behaves the same as that in the iid case. For instance, $S_n$ and $M_n$ are asymptotically independent  if $\{X_n\}$ satisfies certain strong mixing condition and has finite variance; see \cite{anderson1991joint, hsing1995note}.  However, if $X_n$ has a heavy tail on the positive side, the situation can be delicate: a dependence structure in $\{X_n\}$  can possibly lead to    cancellation of large terms in $S_n$, thus breaking down the asymptotic dependence observed in the iid case. Nevertheless,   $(S_n,M_n)$ is expected to still have the same limit distributional   behavior as in the iid case if such a cancellation  is ruled out \cite{anderson1995sums}. Recently, Krizmanić \cite{krizmanic2020joint} investigated the functional convergence of properly normalized      $( S_n(t) , M_n(t))_{t\ge 0} $  under  conditions which implied that clusters of extreme values can be broken down into asymptotically independent blocks,  and that all extremes within each cluster  have the same sign, and hence ruling out the aforementioned cancellation. 

There have been few studies on the joint asymptotic distribution  $(S_n,M_n)$ under strong dependence in $\{X_n\}$, or say \emph{long-range dependence}, except for $\{X_n\}$ being Gaussian. It is known that $1/\log n$ is the critical decay rate of the covariance function, which affects the asymptotic behavior of $(S_n,M_n)$. In particular, $S_n$ and $M_n$ are asymptotically independent when   $\mathrm{Cov}(X_{n+1},X_1)=o(1/\log n)$    with some additional regularity conditions   \cite{ho1996asymptotic, ho1999asymptotic}. If  $ \lim_{n\to\infty} \mathrm{Cov}(X_{n+1},X_1) \log n \in (0,\infty]$, then $S_n$ and $M_n$ are asymptotically dependent and the joint limit was obtained in \cite{ho1996asymptotic, ho1999asymptotic, mccormick2000asymptotic}. 

 On the other hand, the joint sum-max limit theorem for long-range dependent $\{X_n\}$ with heavy tails has not been considered as far as we know. It is worth clarifying  what we mean by ``long-range dependence'' more precisely here. In particular, the  boundary between short and long-range dependence of $\{X_n\}$  can be different for sum and for maximum. We shall follow the phase-transition perspective of \cite{samorodnitsky2016stochastic} towards long-range dependence, that is, the   normalization needed for obtaining non-degenerate limit  is of different order and the limit may also be of different nature compared to the iid case.  Typically, long-range dependence for sum is easier to occur compared to long-range dependence for maximum.  For instance, in the Gaussian case, long-range dependence for sum already takes place when the covariance  $\mathrm{Cov}(X_{n+1},X_1)$ decays like $n^{-\rho}$ with $\rho\in (0,1)$  (see, e.g., \cite{taqqu1975weak}), while it occurs for maximum only when $\mathrm{Cov}(X_{n+1},X_1)$   decays like $1/\log n$, i.e., the critical rate mentioned in the previous paragraph. In the heavy-tailed case, similar phenomenon also appears. For example, for certain stationary heavy-tailed $\{X_n\}$,  under a nonstandard normalization one could obtain the so-called linear fractional stable motion as a limit for sum, whereas the normalization order and the limit for maximum are still the same as the iid case; see, e.g., \cite[Theorems 9.5.7 \& 9.8.1]{samorodnitsky2016stochastic}. 
 In   this paper, we stress that long-range dependence is meant for both sum and maximum. 

In particular, we consider a class of stationary infinitely divisible processes $\{X_n\}$ with regularly varying tails whose dependence structure is determined by a null-recurrent Markov chain governed by a memory parameter $\beta\in(0,1)$. Models of this nature were first introduced in \cite{rosinski1996classes}  and have generated considerable interest since then.  These processes are regarded as having long-range dependence from the perspective of both sum and maximum.  They are associated with  the so-called conservative flows in ergodic theory; we refer to \cite{samorodnitsky2016stochastic} and the references therein for using the likewise concepts in ergodic theory to describe memory properties for stationary sequences. The asymptotic distribution of  the partial  sum process   $(S_{n}(t))_{t\geq 0}$    when $X_n$ has infinite variance was investigated in \cite{owada2015functional}. The limit theorem for the partial maximum process  $(M_{n}(t))_{t\geq 0}$, on the other hand, was initially considered in \cite{owada2015maxima} with a restricted  $\beta$ range.     Then \cite{lacaux2016time} pointed out that looking at  the max limit theorem from the perspective of random sup measures (cf.\ \cite{o1990stationary}) is more revealing. Put simply, such a perspective requires  examining  $M_n(B):=\max_{k/n\in B} X_k$   for different subsets  $B$ of $[0,\infty)$, whereas the partial maximum process  $M_{\floor{nt}}$ can be viewed  as restricting to subsets of the form $B=(0,t]$, $ t> 0$. From now on, the notation  $M_n$ shall always be understood as such a random sup measure. Samorodnitsky and Wang \cite{samorodnitsky2019extremal} took the random sup measure perspective   and established a limit theorem with full range of $\beta\in (0,1)$.

In this paper, we focus on the joint limit of normalized $(S_{n}(\cdot) ,M_n(\cdot) )$ in the space $D[0,1]\times SM[0,1]$, where $D[0,1]$ is the Skorokhod space equipped with $J_1$ topology and $SM[0,1]$ is the space of $\overline{\mathbb{R}}$-valued sup measures on $[0,1]$ equipped with the sup vague topology (\cite{vervaat1988random}).   We show that when $X_n$ has finite variance, $S_n$ and $M_n$ are asymptotically independent; when $X_n$ has infinite variance, $S_n$ and $M_n$ are asymptotically dependent with a dependence structure described by the local time and the range of a shifted stable subordinator (cf.\ \cite{bertoin1999subordinators}).

The paper is organized as follows. In Section \ref{sec2}, we  describe in detail our model $\{X_n\}$ and the limit objects. The main theorems on joint convergence are stated in Section \ref{sec3}. The proofs are included in Section \ref{sec5}.

\section{The setup}\label{sec2}

\subsection{The model and assumptions}\label{subsec1}

We follow essentially the same model as in \cite{samorodnitsky2019extremal}. Let $\{Y_n\}_{n\in\mathbb{N}_0}$ be an irreducible aperiodic null-recurrent Markov chain on $\mathbb{Z}$ with $\mathbb{N}_0 = \{0\}\cup \mathbb{N}=\{0,1,2,\ldots\}$. The trajectory of the Markov chain $(x_0,x_1,\cdots)$ belongs to the space 
{$(\mathsf{E}, \mathcal{E}) := (\mathbb{Z}^{\mathbb{N}_0},\mathcal{C}(\mathbb{Z}^{\mathbb{N}_0}))$,} 
where $\mathcal{C}(\mathbb{Z}^{\mathbb{N}_0})$ denotes for the cylindrical $\sigma$-field. Let $(\pi_i)_{i\in\mathbb{Z}}$ be the unique invariant measure on $\mathbb{Z}$ such that $\pi_{0} = 1$. Let $P_i$ denote the probability measure on {$(\mathsf{E}, \mathcal{E})$}  given by the law of the Markov chain starting at $Y_0 = i$. Define an infinite $\sigma$-finite measure {$\mu$} on {$(\mathsf{E}, \mathcal{E})$} by
$\mu({E}) :=\sum_{i\in\mathbb{Z}}\pi_iP_i({E})$, ${E} \in\mathcal{E}$. 
Let $T:\mathsf{E}\to  \mathsf{E}$ denote the left shift operator $T(x_0,x_1,\cdots) = (x_1,x_2,\cdots)$. Then $\mu$ is $T$-invariant. We consider a stationary process  of the following form: for a suitable measurable function $f: \mathsf{E} \rightarrow \mathbb{R}$, let
\begin{equation}
X_n=\int_{{\mathsf{E}}} f\circ T^{n}(s)M(ds),~~n\in\mathbb{N},\label{eq1}
\end{equation}
 where $M$ is a homogeneous symmetric infinitely divisible random measure on $({\mathsf{E}}, \mathcal{E})$ without a Gaussian component   which is characterized by 
{\[
	\E[e^{i \theta M(E)}]=\exp\left(-\mu(E)\int_{\mathbb{R}} (1-\cos(\theta y)) \rho(dy) \right), \quad E\in \mathcal{E}, \ \mu(E)<\infty,\ \theta\in \mathbb{R},
	\]}
where  $\rho$ on $(\mathbb{R},\mathcal{B}({\mathbb{R}}))$   is  a   Lévy measure satisfying $\rho(\{0\})=0$, $\int_{\mathbb{R}} (1\wedge x^2)  \rho(dx)<\infty $ and 
{$\rho(-S)=\rho(S)$ for any $S\in \mathcal{B}(\mathbb{R})$}, 
where $\mathcal{B}$ stands for Borel $\sigma$-field.  The resulting sequence $\{X_n\}$ is stationary, infinitely divisible and has symmetric marginal distribution.   See \cite[Chapter 3]{samorodnitsky2016stochastic} for more details about infinitely divisible processes, random measures, integrals and their relations. 

Write $\mathbf{x}:=(x_0,x_1,\cdots)$, and set $A:= \{ \mathbf{x}\in {\mathsf{E}} : x_0 = 0\}$ which satisfies $\mu(A)=\pi_0=1$. For the integrand function $f$ in \eqref{eq1}, we shall for simplicity assume   as in \cite{samorodnitsky2019extremal} that  
\[
f=\mathbf{1}_A.
\]
 Define the first entrance time of $A$ as
\begin{equation*}\label{eq:first entrance time}
\varphi_{A} (\mathbf{x}) := \inf\{n\in\mathbb{N}: x_n =0\},~~\mathbf{x}\in {\mathsf{E}}.
\end{equation*}
Throughout the paper, we assume that $P_{0} (\varphi_{A}> n)\in\text{RV}_\infty(-\beta)$ for some $\beta\in(0, 1)$. Here and below, $\text{RV}_\infty(-\beta)$ denotes the class of regularly varying functions/sequences at infinity with exponent $-\beta$. Similarly  $\text{RV}_0$  is used for regular variation at zero. The regular variation assumption on $P_{0}(\varphi_{A}> n)$ can also be expressed in terms of the wandering rate sequence defined by
\begin{equation}
w_n := \mu\left(\bigcup_{k=1}^{n} T^{-k}A\right),~~ n \in\mathbb{N}.\label{def1}
\end{equation}
It follows that
\begin{equation}
w_n=\sum^{n}_{k=1}P_{0} (\varphi_{A}\geq k)\sim \frac{nP_{0}(\varphi_{A}> n)}{1-\beta}\in \text{RV}_\infty(1-\beta);\label{ass1}
\end{equation}
see \cite{owada2015functional}. In addition, we assume that the   Lévy measure $\rho$ of $M$ satisfies for $x>0$, 
\begin{equation}
\rho\left((x,\infty)\right) \in \text{RV}_\infty(-\alpha)\label{ass2}
\end{equation} 
for some $\alpha>0$.   Furthermore, if $0<\alpha<2$, following \cite{bai2020functional}, we assume 
\begin{equation}
\rho\left((x,\infty)\right)= O(x^{-\alpha_0})~~\text{as }x\to 0,\label{eqasssum}
\end{equation} 
for some $\alpha_0\in (0,2)$. 

\subsection{The marginal limits}

In this section we review the known results about  limit theorems of $\{X_n\}$ in \eqref{subsec1}.
The limit  objects  are related to stable subordinators, the details of which we refer to \cite{bertoin1999subordinators}. Recall for $\beta\in(0,1)$, the standard $\beta$-stable subordinator $(S_\beta(t),t \geq 0)$ is defined as  a  Lévy process starting at the origin with Laplace exponent given by $\Phi(\lambda)=\lambda^\beta$, that is, $\mathbb{E}e^{-\lambda S_\beta (t)}= \exp\{-t\lambda^\beta\}$ for $\lambda\geq 0$ and $t \geq 0$. The process $S_\beta(t)$ is right-continuous and strictly increasing almost surely. The inverse stable subordinator, also known as  the Mittag-Leffler process, is   defined by
\begin{equation*}
M_\beta(t) := S^\leftarrow_\beta (t) = \inf\{u\geq 0 : S_\beta(u)\geq t\},~~ t \geq 0.
\end{equation*}
The process $M_\beta(t)$ is   continuous and non-decreasing almost surely, and the left-continuous inverse $S^\leftarrow_\beta (t)$ above can also be replaced with the right-continuous inverse {$S^\rightarrow_\beta(t)=\inf\{u\geq 0 : S_\beta(u)> t\}$}.    The closure of the range of $(S_\beta(t),\ t \geq 0)$ is called the $\beta$-stable regenerative set, denoted by $R^{(\beta)}$, which is regarded as a random element taking value in $\mathfrak{F}([0,\infty))$, the space of closed subsets of $[0,\infty)$ equipped with the Fell topology; see \cite{molchanov2005theory}  for  the details of   random sets theory.

Let throughout  $\{X_k\}$ be the infinitely-divisible stationary process   as described in Section \ref{subsec1}.
Define the partial sum process
\begin{equation}
S_n(t):=\sum_{k=1}^{\floor {nt}}X_k,~~t\geq 0.\label{sum}
\end{equation}
 It is known that the limit of normalized $S_n(t)$ depends on the tails of $X_k$. 

If $X_k$ has finite variance, which is equivalent to existence of the second moment of the L\'evy measure:
\begin{equation}
\int_{-\infty}^\infty x^2\rho(dx)<\infty,\label{ass3}
\end{equation} 
then
\begin{equation}
\frac{1}{nw_n^{-1/2}}(S_n(t))_{t\geq 0}\Rightarrow c_\beta (B_H(t))_{t\geq 0}~~\text{as }n\to\infty\label{eqancon}
\end{equation} 
 on $D[0,\infty)$ equipped with   the Skorokhod $J_1$ topology, where $w_n$ is defined by (\ref{def1}), 
\begin{equation}
c_\beta^2=\int_{-\infty}^\infty x^2\rho(dx)\frac{\Gamma(1+2\beta)}{\Gamma(2-\beta)\Gamma(2+\beta)}\mathbb{E}\left(\{S_\beta(1)\}^{-2\beta}\right),\label{con}
\end{equation} 
and $B_H(t)$ is the fractional Brownian motion with $H =(1+\beta)/2$, satisfying $\mathbb{E}\{B_H(t)^2\}= t^{2H}/2, t\geq 0$; see \cite[ Theorem 9.4.7]{samorodnitsky2016stochastic}.  Notice that in \eqref{eqancon} the normalization sequence $(nw_n^{-1/2})\in \text{RV}_\infty(H)$  in view of \eqref{ass1}, and is hence of higher order compared to $n^{1/2}$ for iid case.

When $X_k$ has infinite variance  satisfying \eqref{ass2} {with $0<\alpha<2$ and \eqref{eqasssum}}, define
\begin{equation}
\rho^{\leftarrow}(y):= \inf\{x > 0 : \rho\left((x,\infty)\right)\leq y\}, ~~y > 0,\label{inv}
\end{equation}
  the generalized inverse of the tail of the   L\'{e}vy measure  which belongs to $\mathrm{RV}_0(-1/\alpha)$. 
 Then we have
\begin{equation*}
\frac{1}{\rho^{\leftarrow}(w_n^{-1})nw_n^{-1}}(S_n(t))_{t\geq 0}\Rightarrow \left(\Gamma(2-\beta)\right)^{-1} C_\alpha^{-1/\alpha} (Y_{\alpha,\beta}(t))_{t\geq 0}~~\text{as }n\to\infty
\end{equation*} 
in the Skorokhod space $D[0,\infty)$ equipped with the  $J_1$ topology, where  the limit process    $Y_{\alpha,\beta}$ is a  symmetric $\alpha$-stable (S$\alpha$S) process which will be described below, the normalization   sequence $\left(\rho^{\leftarrow}(w_n^{-1})nw_n^{-1}\right) \in \mathrm{RV}(\beta+(1-\beta)/\alpha)$, and
\begin{equation} 
C_\alpha=\begin{cases}
\left(\Gamma(1-\alpha)\cos(\pi\alpha/2)\right)^{-1}&\text{ if }\alpha\neq 1,\\
2/\pi &\text{ if }\alpha= 1; 
\end{cases}\label{stacon}
\end{equation} 
see \cite[{Theorem 5.1 \& Example 5.5}]{owada2015functional} as well as \cite{bai2020functional}.                                                                                                             To describe the limit process $Y_{\alpha,\beta}$, $0 <\alpha< 2$, $0 <\beta< 1$, introduce a probability space $(\Omega',\mathcal{F}',P')$  different from the underlying probability space.  Let $(M_\beta(t,\omega'))_{t\ge 0}$ be a Mittag-Leffler process defined on $(\Omega',\mathcal{F}',P')$, $\omega'\in \Omega'$. Suppose $\nu$ is a measure on $[0,\infty)$ given by $\nu(dx) = (1-\beta)x^{-\beta} dx$, $x > 0$.   Now define  
\begin{equation}
Y_{\alpha,\beta}(t) :=\int_{\Omega'\times[0,\infty)}M_\beta\left((t- x)_+,\omega'\right)dZ_{\alpha,\beta}(\omega',x),~~ t\geq 0,\label{sumlim}
\end{equation} 
where $Z_{\alpha,\beta}$ is a S$\alpha$S   random measure on $\Omega'\times[0,\infty)$ with control measure $m:=P'\otimes \nu$, characterized by $\E \exp(i\theta  Z_{\alpha,\beta}(A))=\exp(-|\theta|^{\alpha}m(A))$ for $\theta\in \mathbb{R}$, 
{$A\in \mathcal{F}'\otimes\mathcal{B}([0,\infty))$ }
with $m(A)<\infty$.    

Next we review the limit theorems for the   maximum. For any $n\in\mathbb{N}$, introduce  
\begin{equation}
M_n(B):=\max_{k\in (n B)\cap \bb{N}}X_k,~~ B\in\mathcal{G}([0,\infty)),\label{max}
\end{equation}
with the understanding     $\max_\emptyset:=-\infty$, and  $\mathcal{G}([0,\infty))$ denotes the collection of open subsets of $[0,\infty)$  under the subspace topology. If (\ref{ass2}) holds for some $\alpha>0$ and the following regularity condition holds:
\begin{equation}
\sup_{n\in\mathbb{N}}\frac{nP_{0} (\varphi_{A}=n)}{P_{0} (\varphi_{A}> n)}<\infty,\label{eqassmax}
\end{equation}
then    \cite{samorodnitsky2019extremal} established the following limit theorem:
\begin{equation*}
\frac{1}{\rho^{\leftarrow}(w_n^{-1})}(M_n(B))_{B\in \mathcal{G} ([0,\infty))}\Rightarrow (\eta_{\alpha,\beta}(B))_{B\in \mathcal{G} ([0,\infty))}~~\text{as }n\to\infty,\label{eqbncon}
\end{equation*}
 in the sup measure space $SM[0,\infty)$ equipped with the sup vague topology, where the limit random sup measure $\eta_{\alpha,\beta}$ will be defined below in \eqref{limsm}; we refer to   \cite{vervaat1988random,o1990stationary} for the definition  of the sup vague topology and the  details of the theory of  random sup measures. In the limit theorem above,  the normalization   sequence $(\rho^{\leftarrow}(w_n^{-1})) \in \mathrm{RV}_\infty ((1-\beta)/\alpha)$.
To define   $\eta_{\alpha,\beta}$,  consider a Poisson point process on the product space
$[0,\infty) \times [0,\infty) \times \mathfrak{F}([0,\infty))$ with mean measure
$\alpha u^{-(1+\alpha)} du(1-\beta)v^{-\beta} dvdP_{R^{(\beta)}}$, where $P_{R^{(\beta)}}$ is the law of the $\beta$-stable regenerative set. Let $(U_j , W_j , F_j )_{j\in\mathbb{N}}$ denote a measurable enumeration of the points of the point process, and define
$\widetilde{F}_j:=W_j +F_j:=\{W_j+x:\ x\in F_j\}$, $j \in\mathbb{N}$.
For   $B\in \mathcal{G} ([0,\infty))$, the limit random sup measure is determined by  
\begin{equation}\label{limsm}
\eta_{\alpha,\beta}(B)=\sup_{t\in B} \left( \sum_{j=1}^\infty U_j \mathbf{1}_{\{t\in\widetilde{F}_j\}}\right),~~ B\in \mathcal{G} ([0,\infty)).
\end{equation}
It is known that for $\beta\in (0,1/2]$, the random variable $\eta_{\alpha,\beta}(B)$ marginally follows  an $\alpha$-Fréchet distribution, which is no longer the case if $\beta\in (1/2,1)$.

\section{Main results}\label{sec3}

Throughout we assume that {$\{X_n\}$} is the stationary infinitely-divisible process with symmetric and regularly varying tails as described in Section \ref{subsec1}.  For simplicity, we formulate our results restricting to the unit time interval $[0,1]$, although they can also be   extended to the half interval $[0,\infty)$.

\subsection{Joint convergence of processes with finite variance}\label{subsec2}

When the variance of $X_n$ is finite, we show that the sum $S_n(t)$ and the maximum $M_n(B)$ are asymptotically independent as described by the   theorem below. Let $\cl{G}([0,1])$ denote the open subsets of $[0,1]$ with respect to the subspace topology inherited from $\mathbb{R}$.

\begin{theorem}\label{thmfv}
	Let $\{X_n\}_{n\in\mathbb{N}}$ be a stationary process defined by $(\ref{eq1})$ with $f=\mathbf{1}_A$. Let $(S_n(t))_{t\in [0,1]}$ be defined by $(\ref{sum})$ and $(M_n(B))_{B\in \cl{G}([0,1])}$ be defined by $(\ref{max})$. Suppose the  regular variation  conditions $(\ref{ass1})$ and   $(\ref{ass2})$ hold for $0<\beta<1$ and $\alpha\ge 2$, respectively. Assume the finite variance condition $(\ref{ass3})$ and the regularity condition  $(\ref{eqassmax})$. Then
	\begin{equation*}
	\begin{pmatrix}
	(n^{-1}w_n^{1/2}S_n(t))_{t\in[0,1]}\\
	(b_n^{-1}M_n(B))_{B\in \cl{G}([0,1])}
	\end{pmatrix}\Rightarrow\begin{pmatrix}
	(c_\beta B_H(t))_{{t\in[0,1]}}\\
	(\eta_{\alpha,\beta}(B))_{B\in \cl{G}([0,1])}
	\end{pmatrix} ~~\text{as }n\to\infty
	\end{equation*} 
	weakly in the  product space $D[0,1]\times SM[0,1]$, with $D[0,1]$  equipped with the Skorokhod  $J_1$ topology and $SM[0,1]$ equipped with the sup vague topology, where $w_n$ is defined by $(\ref{def1})$, $b_n=\rho^{\leftarrow}(w_n^{-1})$ for $\rho^{\leftarrow}$ defined by $(\ref{inv})$ and $c_\beta$ is defined by $(\ref{con})$, $B_H$ is the fractional Brownian motion with $H =(1+\beta)/2$  satisfying $\mathbb{E}\{B_H(t)^2\}= t^{2H}/2, t\geq 0$,  and $\eta_{\alpha,\beta}$ is a random sup measure defined by $(\ref{limsm})$. Furthermore, $B_H$ and $\eta_{\alpha,\beta}$ in the limit are independent.
\end{theorem}
The proof of Theorem \ref{thmfv} can be found in Section \ref{sec:pf finite var}. We shall for simplicity omit treating the possible  case with $\alpha=2$ and  $X_n$ having an infinite variance. Such a subtle case requires a modification of the normalization $n^{-1}w_n^{1/2}$ by an additional diverging slowly varying factor. Other than this  change, we expect the limit theorem is of the same nature as Theorem \ref{thmfv}, namely, a fractional Brownian motion limit still appears in the limit for the sum component, although this case is not covered by \cite[Theorem 9.4.7]{samorodnitsky2016stochastic}.

\subsection{Joint convergence of processes with infinite variance}\label{subsec3}

When $0<\alpha<2$, we shall show that the sum $S_n(t)$ and the maximum $M_n(B)$ are asymptotically dependent. We first introduce some necessary ingredients for describing the joint limit.

For each subordinator $\sigma$, define the (right-continuous) inverse process, or say the local time process,   as 
\begin{equation}
L_{\sigma}(x) =  \inf\{t \geq 0 : \sigma(t) > x\},~~ x \geq 0\label{eqlocal}
\end{equation}
and the closed range of $\sigma$ as 
\begin{equation}
R_{\sigma} = \overline{\{\sigma(t): t\ge 0\} },\label{eqrange}
\end{equation}
where the bar above denotes the closure.

To express the joint limit, it is   convenient to use the so-called series representations (cf.\ \cite[Section 3.4]{samorodnitsky2016stochastic}). Let {$\{\varepsilon_j\}_{j\in\mathbb{N}}$} be a sequence of iid Rademacher random variables and {$\{\Gamma_j\}_{j\in\mathbb{N}}$} be an ordered sequence of standard Poisson arrivals on $[0,\infty)$.  
Let $\{\sigma_j\}_{j\in\mathbb{N}}$ be  a sequence of iid\ copies of standard $\beta$-stable subordinators and $\{V_j\}_{j\in\mathbb{N}}$ be a  sequence of iid\ random variables in $[0,1]$ with common distribution 
\begin{equation*}\label{eq:V CDF}
\mathbb{P}(V\leq x) = x^{1-\beta},  \quad x\in[0,1].
\end{equation*} All aforementioned random sequences   are assumed to be independent of each other.
Denote the local time process and the closed range of $\sigma_j$ by $L_j=L_{\sigma_j}$ and $R_j=R_{\sigma_j}$, respectively. Let $\widetilde{R}_j:=V_j+R_j:=\{ V_j+x:\ x\in R_j\}$ be the random set $R_j$ shifted by $V_j$. Note that $0\in R_j$ and hence $V_j\in \widetilde{R}_j$. Introduce the series representations
\begin{equation}
S(t):=(2{ C_\alpha})^{1/\alpha}\sum_{j=1}^\infty\varepsilon_j\Gamma_j^{-1/\alpha}L_{j}\left((t-V_{j})_+\right),~~t\in[0,1],\label{limsum}
\end{equation}
where $C_\alpha$ is as in  \eqref{stacon}. {The series in \eqref{limsum} converges almost surely (cf.\ \cite{bai2020functional}; see also \cite[Section 3.4]{samorodnitsky2016stochastic} for general information about series representation of infinitely divisible processes)}. Introduce
 the random sup measure 
\begin{equation}\label{limmax}
M(B):=\sup_{t\in B} \left( \sum_{j=1}^\infty\Gamma_j^{-1/\alpha}\mathbf{1}_{\{t\in\widetilde{R}_j\}} \right),\quad B\in \cl{G}([0,1]).
\end{equation}
Define  
\begin{equation}\label{eq:I_S}
I_S:=\begin{cases}
\bigcap_{j\in S}\widetilde{R}_j\cap[0,1],\quad  &\emptyset \neq S\subset\{1,2,\cdots\} \\
 [0,1], \quad &S=\emptyset. \
\end{cases}
\end{equation}  
With $|S|$ denoting the cardinality of $S$,
it is known that  the intersection   $I_S $ above is nonempty almost surely only when $1\le |S|\le \ell_\beta$, $\ell_\beta:=\max\{\ell\in\mathbb{N}:\ell<(1-\beta)^{-1}\}$; see \cite[Corollary B.3]{samorodnitsky2019extremal}.
An equivalent representation of $M$ which is useful in the proofs is given by 
\begin{equation}
M(B)=\sup_{S\subset\{1,2,\cdots\} }\mathbf{1}_{\{I_S\cap B\neq\emptyset\}}\sum_{j\in S}\Gamma_j^{-1/\alpha},~~B\in \cl{G}([0,1]),\label{maxidx}
\end{equation}
where the equality holds almost surely for any   $B\in \cl{G}([0,1])$. {In \eqref{maxidx} almost surely the   supremum   can be attained at some (random) subset $S$   with $|S|\le \ell_\beta$.}
See   \cite{samorodnitsky2019extremal} and also \cite[Proposition 3.1]{bai2021phase}.
Now we state our main result in the infinite variance case.

\begin{theorem}\label{thmmain}
	Let $\{X_n\}_{n\in\mathbb{N}}$ be a stationary process defined by $(\ref{eq1})$ with $f=\mathbf{1}_A$. Suppose $(\ref{ass1})$ and $(\ref{ass2})$ hold for $0<\beta<1$ and $0<\alpha<2$, respectively. Let $(S_n(t))_{t\in [0,1]}$ be defined in $(\ref{sum})$ and $(M_n(B))_{B\in \cl{G}([0,1])}$ be defined in $(\ref{max})$. Assume that $(\ref{eqasssum})$ holds for some $\alpha_0<2$ and $(\ref{eqassmax})$ holds. Then
	\begin{equation*}
	\begin{pmatrix}
	(c_n^{-1}S_{n}(t))_{t\in [0,1]}\\
	(b_n^{-1}M_{n}(B))_{B\in \cl{G}([0,1])}
	\end{pmatrix} \Rightarrow\begin{pmatrix}
	(C_\alpha^{-1/\alpha}S(t))_{t\in [0,1]}\\
	(M(B))_{B\in \cl{G}([0,1])}
	\end{pmatrix} ~~\text{as }n\to\infty\label{thmmaineq1}
	\end{equation*}
	weakly in $D[0,1]\times SM[0,1]$, where $c_n=\rho^{\leftarrow}(w_n^{-1})nw_n^{-1}/\Gamma(2-\beta)$, $b_n=\rho^{\leftarrow}(w_n^{-1})$ with $w_n$ as in $(\ref{def1})$ and $\rho^{\leftarrow}$ as in $(\ref{inv})$, $C_\alpha$ is as in  \eqref{stacon}, and  the limit process $S$ and random sup measure $M$ are as in $(\ref{limsum})$ and $(\ref{limmax})$.
\end{theorem}
The proof of Theorem \ref{thmmain} can be found in Section \ref{sec:pf inf var}.

\begin{remark}
	The limit process $Y_{\alpha,\beta}$ described by  the stochastic integral in $(\ref{sumlim})$ when restricted to time interval $[0,1]$ is equal in law to the  series representation $S$ in \eqref{limsum}; see \cite[Example 3.4.4]{samorodnitsky2016stochastic}, and also  \cite{bai2020functional}.      The random sup measure $\eta_{\alpha,\beta}$ defined in $(\ref{limsm})$ restricted on $[0,1]$ has the same law with the random sup measure $M$ defined in $(\ref{limmax})$; 
	see   equation (3.7) of \cite{samorodnitsky2019extremal}.
\end{remark}

The following proposition confirms that in contrast to Theorem \ref{thmfv}, the limits in Theorem \ref{thmmain} are dependent.
\begin{proposition}\label{Pro:dep}
	$(S(t))_{t\in[0,1]}$ defined in $(\ref{limsum})$ and $(M(B))_{B\in\mathcal{G}([0,1])}$ defined in $(\ref{limmax})$ are dependent. 
\end{proposition}
The proof of Proposition \ref{Pro:dep} can be found in Section \ref{sec:dep}.

\subsection{Joint convergence of subordinators with their local times and ranges}\label{subsec4}

Here, we provide a  result key to  the proof of Theorem \ref{thmmain}, which  might be of independent interest. It states that the weak convergence of strictly increasing subordinators implies its joint weak convergence with its local time and  range. We also include an independent shift for our purpose.

\begin{proposition}\label{thmsub}
	Let  $\sigma$, and $\sigma_n$, $n\in\mathbb{N}$, each be subordinator, with  $\sigma$   strictly increasing on $[0,\infty)$ almost surely. Let $V$, $V_n$, $n\in\mathbb{N}$, be nonnegative random variables. Suppose  $\sigma_n$   is independent of $V_n$,   $n\in \bb{N}$, and $\sigma$ is independent of $V$. Let $L$ and $L_n$ be the local times of $\sigma$, $\sigma_n$ respectively as defined by $(\ref{eqlocal})$, and let $R$ and $R_n$ be the closed ranges of $\sigma$ and $\sigma_n$ respectively as defined by $(\ref{eqrange})$. If $\sigma_n(1)\stackrel{d}{\to}\sigma(1)$, and $V_n\stackrel{d}{\to}V$ as $n\to\infty$, then for any  $t,y>0$, we have
	\begin{equation*}
		\begin{pmatrix}
			(V_n+\sigma_n(s))_{s\in [0,t]}\\
			\left(L_n\left((x-V_n)_+\right)\right)_{x\in [0,y]}\\
			V_n+R_n
		\end{pmatrix} 
		\Rightarrow\begin{pmatrix}
		(V+\sigma(s))_{s\in [0,t]}\\
		\left(L\left((x-V)_+\right)\right)_{x\in [ 0,y]}\\
		V+R
	\end{pmatrix}~~\text{as }n\to\infty
	\end{equation*} 
	weakly in the product space   $D[0,t]\times D[0,y]\times\mathfrak{F}([0,\infty))$, where $D[0,t]$ and $D[0,y]$ are  equipped with the uniform topology, and  $\mathfrak{F}([0,\infty))$  is equipped with the Fell topology. In particular,
	\begin{equation*}
		\begin{pmatrix}
			(\sigma_n(s))_{s\in [0,t]}\\
			\left(L_n(x)\right)_{x\in [0,y]}\\
			 R_n
		\end{pmatrix} 
		\Rightarrow\begin{pmatrix}
		( \sigma(s))_{s\in [0,t]}\\
		\left(L( x )\right)_{x\in [ 0,y]}\\
		 R
	\end{pmatrix}~~\text{as }n\to\infty
	\end{equation*} 
	in   $D[0,t]\times D[0,y]\times\mathfrak{F}([0,\infty))$.
\end{proposition}
 
\begin{remark}
 Albeit a natural result, we are unable to find explicit  and general statements in literature about the  joint convergence of subordinators with their local times and ranges. See   \cite[Proposition A.4]{chen2021new} for a similar result regarding a joint convergence of the inverse of certain positive random walk and its range, which   could   be applied as  alternative arguments  in the proof of Theorem \ref{thmmain}. Note that we have   restricted the limit subordinator $\sigma$ to be strictly increasing, i.e., $\sigma$ is not a compound Poisson process. Since the latter can be  intrinsically treated as a discrete process   and rarely appears as a weak limit, we omit treating that case. 
\end{remark}
 
\section{Proofs}\label{sec5}

\subsection{Criterion for weak convergence}\label{eq:cri weak conv}

It is a classical recipe to establish weak convergence of stochastic processes   through proving convergence of finite-dimensional distributions (fdd) and tightness; see for instance, the treatment of weak convergence in the Skorokhod space in \cite[Chapter 3]{billingsley1999convergence}.  In this paper,   we are also dealing with random elements whose codomain is beyond the   Skorokhod space, i.e.,   the sup measure space $SM[0,1]$. Not suprisingly, the same  ``fdd$+$tightness'' recipe continues to work. We shall clarify this recipe with a general result below as well as its specification to our context.
Below we use $\mathcal{B}( {\mathsf{E}})$ to denote the Borel $\sigma$-field on a topological space ${\mathsf{E}}$.

Suppose $T$ is the (time) index set for the processes, and $\mathcal{U}$ is a separable metric space acting as the state space of the processes.
Let $\mathcal{M}$ be a  subspace of functions   $x$: $T\rightarrow \mathcal{U}$ which forms a Polish space, i.e., it admits a complete and separable metric.  
Define the  (multi-)projection mapping $\pi_{t_1,\ldots,t_d}:\mathcal{M}\rightarrow \mathcal{U}^d$, $\pi_{t_1,\ldots,t_d} x =(x(t_1),\ldots,x(t_d))$ for $x\in \mathcal{M}$ and $t_1,\ldots,t_d\in T$,  $d\in\mathbb{N}$. We assume throughout that each projection mapping $\pi_t:\cl{M}\mapsto \mathcal{U}$   is measurable, $t\in T$.

\begin{example}\label{eg:1}
	For our application, recall that both $D[0,1]$ with $J_1$ topology and $SM[0,1]$ with the sup vague topology are metrizable with a complete separable metric (\cite[Section 12]{billingsley1999convergence} and \cite[Remark 5.6]{vervaat1988random}). Now set $T=[0,1]\times \mathcal{G}([0,1]) $,  and let $\mathcal{M}=D[0,1]\times SM[0,1]$   equipped with the product metric,  and $\mathcal{U}=\mathbb{R}\times \overline{\mathbb{R}}$. The measurability of the projection mapping  follows from \cite[Section 12]{billingsley1999convergence}  and the definition of sup vague topology (see also the proof of \cite[Theorem 11.1]{vervaat1988random}).
\end{example}

For a stochastic process $\xi=(\xi(t))_{t\in T}$ which takes values in $\mathcal{M}$ with                                                                                                                  law $P_\xi$ on $\mathcal{M}$, we define  the following index subset of continuity points: 
\[
T_\xi =\left\{t\in T:\  P_\xi\left(\pi_t:  \cl{M}\mapsto \mathcal{U}  \text{ is continuous}\right)=1 \right\}.
\]
Following \cite{vervaat1988random}, we introduce the following concepts. Below $\eqd$ denotes equality in law.

\begin{definition}\label{Def:conv det}

 A subset of indices $T_0\subset T$ is said to be law determining if the following holds: for two processes $\xi_1$ and $\xi_2$ taking value in $\mathcal{M}$, whenever  $\pi_{t_1,\ldots,t_d}\xi_1 \eqd \pi_{t_1,\ldots,t_d} \xi_2$ on $\mathcal{U}^d $ for all $t_1,\ldots,t_d\in T_0$,  $d\in \mathbb{N}$, we necessarily have $\xi_1\eqd \xi_2$ on $\mathcal{M}$. In other words, the law of any random element $\xi$ taking values in $\mathcal{M}$ is completely determined by the fdds of $(\xi)_{t\in T_0}$.

Furthermore, a subset $T_0\subset T$ is said to be convergence determining, if for any two processes $\xi_1$ and $\xi_2$ taking values in $\mathcal{M}$, the subset $T_0\cap T_{\xi_1}\cap T_{\xi_2}$ is law determining.
\end{definition}

Below we suppose each of the stochastic processes $\xi=(\xi(t))_{t\in T},\xi_1=(\xi_1(t))_{t\in T},\xi_2=(\xi_2(t))_{t\in T},\ldots$,  takes values in $\mathcal{M}$.   
\begin{definition}\label{def:fdd}
 	We say $\xi_n$ converges in fdd to $\xi$  in a subset $T_0\subset T$  as $n\rightarrow\infty$, denoted by $\xi_n\stackrel{fdd}{\to}\xi$ in $T_0$, if  the weak convergence
$$\pi_{t_1,\ldots,t_d}\xi_n \Rightarrow\pi_{t_1,\ldots,t_d}\xi $$
	holds on $\mathcal{U}^d$ as $n\to\infty$ for all  $d\in\mathbb{N}$  and $t_1,\ldots,t_d\in T_0$. When $T_0$ is omitted, we understand $T_0$ as the full index set $T$.  
\end{definition}

Recall (the laws of) the family of random elements $\{\xi_n\}$ is said to be \emph{tight} on $\mathcal{M}$, if  \[\sup_{K}\inf_{n}\bb{P}(\xi_n\in K)=1\] with the sup taken over any compact subset $K$ of $\mathcal{M}$.
\begin{proposition}\label{thmfdd}
Suppose a  convergence determining subset  $T_0\subset T$ in the sense of Definition \ref{Def:conv det} exists.   
Then     the weak convergence $\xi_n\Rightarrow \xi$ on $\mathcal{M}$ as $n\rightarrow\infty$ holds  if and only if $\xi_n\stackrel{fdd}{\to}\xi$ as $n\rightarrow\infty$ in $T_{\xi}$ and   $\{\xi_n\}$ is tight.
 Moreover, if $T_0\subset T_\xi$, one may replace  $T_{\xi}$ by $T_0$ in the preceding statement.
\end{proposition}
\begin{proof}
For the ``if'' part, by tightness, for any subsequence of $\{\xi_n\}$, there exists a further subsequence which converges weakly to a limit law of a random element $\xi^*\in \mathcal{M}$ (e.g., \cite[Theorem 16.3]{kallenberg:2002:foundations}). It is enough to show the limit law $P_{\xi^*}$ of $\xi^*$ coincides with  the limit law $P_\xi$ of $\xi$.  Let the law of each $\xi_n$ be $P_{\xi_n}$, $n\in \mathbb{N}$.  For any   points $t_1,\ldots,t_d\in T_{\xi^*}$, $d\in \mathbb{N}$, we note that also the mapping $\pi_{t_1,\ldots,t_d}:\mathcal{M}\rightarrow \mathcal{U}^d$ is continuous almost surely with respect to $P_{\xi^*}$.
 So by  the continuous mapping theorem (e.g., \cite[Theorem 4.27]{kallenberg:2002:foundations}), we have as $n\rightarrow\infty$ along the same sub-sub sequence, the weak convergence $P_{\xi_n}\pi_{t_1,\ldots,t_d}^{-1}\Rightarrow  P_{\xi^*} \pi_{t_1,\ldots,t_d}^{-1}$. By the assumed fdd convergence and the uniqueness of weak limits, we have $P_{\xi^*} \pi_{t_1,\ldots,t_d}^{-1}=P_\xi \pi_{t_1,\ldots,t_d}^{-1}$ for $t_1,\ldots,t_d\in T_0\cap  T_{\xi}\cap T_{\xi^*} $. Then the desirable conclusion $P_\xi=P_{\xi^*}$ follows from the definition of  convergence determination of $T_0$ in Definition \ref{Def:conv det}.

The ``only if'' part follows from the fact that weak convergence implies tightness  (e.g., \cite[Theorem 16.3]{kallenberg:2002:foundations}),  and the continuous mapping theorem applied to $\pi_{t_1,\ldots,t_d}$ with $t_1,\ldots, t_d\in T_{\xi}$, $d\in \mathbb{N}$.

The second claim about the case where $T_0\subset T_\xi$ follows similarly as above.
\end{proof}

\begin{example}\label{eg:fdd}
In the setup of Example \ref{eg:1},  we claim that
	the subset    $T_0=[0,1]\times \mathcal{I} $ is convergence determining in the sense of Definition \ref{Def:conv det}, where $\mathcal{I}$ is the collection of  all nonempty open sub-intervals of $[0,1]$.  To show this, let $\xi_1=(Z_1,M_1) $ and $\xi_2=(Z_2,M_2)$ be two arbitrary random elements taking values in $\mathcal{M}=D[0,1]\times SM[0,1]$. In view of \cite[Section 12]{billingsley1999convergence} and the proof of \cite[Theorem 12.2]{vervaat1988random}, there exist  $\cl{J}_i \subset [0,1]$  and $\mathcal{I}_i\subset \mathcal{I}$,   $i=1,2$, such that each $[0,1]\setminus \cl{J}_i$ or  $\mathcal{I}\setminus \mathcal{I}_i$  is countable.  In addition, for these subsets selected,
	 one can ensure that a projection evaluated at a point of $\cl{J}_i$ in $D[0,1]$, or at a point of $\mathcal{I}_i$ in $SM[0,1]$, is continuous with respect to the marginal law of $Z_i$ or $M_i$ respectively, $i=1,2$. These imply that  $\cl{J}_i\times \mathcal{I}_i\subset T_{\xi_i}$, $i=1,2$. 
	 Now it suffices to show that  as a subset of $T_0\cap T_{\xi_1}\cap T_{\xi_2}$, the index set
	 \[T^*:=(\cl{J}_1\cap \cl{J}_2)\times (\mathcal{I}_1\cap \mathcal{I}_2),
	 \]   is law determining in the sense of Definition \ref{Def:conv det}. For this, it suffices by Dynkin's $\pi$-$\lambda$ theorem to  show that the $\pi$-system
	 	\begin{equation*} 
	\{ \pi_{t_1 ,\ldots,t_d}^{-1} U,\   d\in\mathbb{N}  ,\ t_1,\ldots,t_d\in T^*,\  U\in \mathcal{B}(\mathcal{U}^d)\}
	\end{equation*} 
	generates $\mathcal{B}(\mathcal{M})$.   To see this, notice that the same claim  holds    when $\mathcal{M}=D[0,1]$, $T^*=\cl{J}_1\cap \cl{J}_2$ and $\cl{U}=\mathbb{R}$  (\cite[Theorem 12.5]{billingsley1999convergence}),  or when $\mathcal{M}= SM[0,1]$ and $T^*=\mathcal{I}_1\cap \mathcal{I}_2$ and  $\cl{U}=\overline{\mathbb{R}}$ (\cite[Theorem 11.1]{vervaat1988random}). Then notice that  the Borel $\sigma$-field on the product space  $\mathcal{B}(D[0,1]\times SM[0,1])$ coincides with the product of Borel $\sigma$-fields $\mathcal{B}(D[0,1])\otimes \mathcal{B}(SM[0,1])$   since both $ D[0,1]$ and $ SM[0,1]$ are separable (\cite[Lemma 1.2]{kallenberg:2002:foundations}).

We also observe that a projection  mapping on $D[0,1]$  is continuous almost surely with respect to the limit law  $B_H$ in Theorem \ref{thmfv} or $S$ in Theorem \ref{thmmain}. This follows from \cite[Theorem 12.5]{billingsley1999convergence}   because both the fractional Brownian motion $B_H$ and the stable process $S$  (see \cite[Theorem 3.3]{owada2015functional}) admit   versions with  continuous paths  almost surely. Furthermore, a projection mapping on $SM[0,1]$ is also continuous with respect to the laws of the limit random sup measures in  Theorems \ref{thmfv}  and \ref{thmmain}; see  the proof of \cite[Proposition 5.2]{samorodnitsky2019extremal}.  Hence the  projection mapping on $D[0,1]\times SM[0,1]$ is continuous almost surely with respect to the law of the joint limit $\xi$  in either Theorem \ref{thmfv}  or \ref{thmmain}.  Hence $T_0=[0,1]\times \mathcal{I}\subset T_\xi$ for either of the theorems. 
	
\end{example}


\subsection{Finite variance case}\label{sec:pf finite var}

\begin{proof}[Proof of Theorem~{\upshape\ref{thmfv}}]
	Observe that the tightness holds on the product space if it holds on each marginal space. 
 The weak convergence, and hence the tightness  of $n^{-1}w_n^{1/2}S_n$ in $D[0,1]$,   follows from \cite[Theorem 9.4.7]{samorodnitsky2016stochastic}. 
 On the other hand, the normalized empirical sup measure
	 $b_n^{-1}M_n$ is automatically tight since the space $SM[0,1]$ is compact (\cite[Theorem 4.2]{vervaat1988random}). By Proposition \ref{thmfdd} and Example \ref{eg:fdd}, we only need to show the convergence in fdd in $[0,1]\times \cl{I}$ in the sense of Definition \ref{def:fdd}, recalling that $\cl{I}$ is the collection of  all nonempty open sub-intervals of $[0,1]$.  
	
	Fix $m>0$. For each $X_k$ in \eqref{eq1}, as   in \cite[Section 5]{samorodnitsky2019extremal}, it can be decomposed as
	\begin{equation*}
		X_k=X_{k,m}^{(1)}+X_{k,m}^{(2)},
	\end{equation*} 
	where
	\begin{equation*}
		X_{k,m}^{(j)} := \int_{{\mathsf{E}}} f\circ T^k({s})M^{(j)}_{m}(d{s}),~~j=1,2.
	\end{equation*} 
	Here $M^{(1)}_m$ and $M^{(2)}_m$ are two independent homogeneous symmetric infinitely divisible random measures. The	  Lévy measure for $M^{(1)}_m$ is the measure $\rho$ for $M$ restricted to the set $\{ | x|\leq m\}$ while the   Lévy measure for $M^{(2)}_m$ is the measure $\rho$ restricted to the set $\{| x| > m\}$. Define for each $j=1,2$,
	\begin{equation*}
		S^{(j)}_{n,m}(t):=\sum_{k=1}^{\lfloor nt\rfloor}X_{k,m}^{(j)},~~t\in[0,1]
	\end{equation*} 
and
\begin{equation*}
	M^{(j)}_{n,m}(B):=\max_{k\in n B\cap\mathbb{N}}X_{k,m}^{(j)},~~B\in \cl{G}([0,1])
\end{equation*} 
 with the understanding     $\max_\emptyset:=-\infty$.  By \cite[Theorem 9.4.7]{samorodnitsky2016stochastic} mentioned above, we have
	\begin{equation*}
		n^{-1}w_n^{1/2}S^{(1)}_{n,m}\stackrel{fdd}{\to} c_{\beta,m}B_H,~~\text{as }n\to\infty,
	\end{equation*} 
in $[0,1]$ as $n\rightarrow\infty$,	where
	\begin{equation*}
		c_{\beta,m}^2=\int_{- m }^m x^2\rho(dx)\frac{\Gamma(1+2\beta)}{\Gamma(2-\beta)\Gamma(2+\beta)}\mathbb{E}\left(\{S_\beta(1)\}^{-2\beta}\right);
	\end{equation*}
  by   a slight extension of \cite[Theorem 5.1]{samorodnitsky2019extremal} to incorporate the more  general regular variation assumption on $\rho$ (see also the proof of Theorem \ref{thmmain} in Section \ref{sec:pf inf var}), we have
	\begin{equation*}
		b_n^{-1}M^{(2)}_{n,m}\stackrel{fdd}{\to} \eta_{\alpha,\beta},~~\text{as }n\to\infty
	\end{equation*}
	in $\cl{I}$. Note that the limit above is the same regardless of the truncation parameter $m$ {as can be read from \cite[Theorem 5.1]{samorodnitsky2019extremal}}.  {An explanation   is that the limit  extremal behavior  only depends  on the tail behavior of the joint distributions of $\{X_{k,m}^{(2)}\}_{k=1,\ldots,n}$, which in turn is only determined by  the tail behavior of the L\'evy measure $\rho$.} Then by independence between $S^{(1)}_{n,m}$ and $M^{(2)}_{n,m}$,  we conclude
	\begin{equation*}
		\begin{pmatrix}
			n^{-1}w_n^{1/2}S^{(1)}_{n,m}\\
			b_n^{-1}M^{(2)}_{n,m}
		\end{pmatrix}\stackrel{fdd}{\to}\begin{pmatrix}
			c_{\beta,m} B_H\\
			\eta_{\alpha,\beta}
		\end{pmatrix}~~\text{as }n\to\infty,
	\end{equation*}
	in $[0,1]\times \cl{I}$, where $B_H$ and $\eta_{\alpha,\beta}$ are independent. Since $c_{\beta,m}\to c_\beta$ as $m\to\infty$, we have
	\begin{equation*}
		\begin{pmatrix}
			c_{\beta,m} B_H\\
			\eta_{\alpha,\beta}
		\end{pmatrix}\stackrel{fdd}{\to}\begin{pmatrix}
			c_{\beta} B_H\\
			\eta_{\alpha,\beta}
		\end{pmatrix},~~
	\end{equation*}
	in $[0,1]\times \cl{I}$ as $m\to\infty$.
	The desirable convergence in fdd    follows   from a standard triangular approximation argument  (e.g.,    \cite[Theorem 3.2]{billingsley1999convergence}), once we show that for any $t\in[0,1]$, $B\in \cl{G}([0,1])$ and any $ {\epsilon}>0$,  
	\begin{equation}
	\lim_{m\to\infty}\limsup_{n\to\infty}\mathbb{P}\left(n^{-1}w_n^{1/2}\left| S_{n}(t)-S^{(1)}_{n,m}(t)\right|>{\epsilon}\right)=0\label{th1eq1}
	\end{equation}
	and
	\begin{equation}
	\lim_{m\to\infty}\limsup_{n\to\infty}\mathbb{P}\left(b_n^{-1}\left| M_{n}(B)-M^{(2)}_{n,m}(B)\right|>{\epsilon}\right)=0.\label{th1eq2}
	\end{equation}
	Note that for fixed $t\in [0,1]$, by again  \cite[Theorem 9.4.7]{samorodnitsky2016stochastic},
	\begin{equation*}
		n^{-1}w_n^{1/2}\left(S_{n}(t)-S^{(1)}_{n,m}(t)\right)=n^{-1}w_n^{1/2}S^{(2)}_{n,m}(t)\stackrel{fdd}{\to} (c_\beta-c_{\beta,m})B_H(t),~~\text{as }n\to\infty,
	\end{equation*}
  from which	(\ref{th1eq1}) follows since $\lim_{m\to\infty}c_{\beta,m}=c_\beta$. To prove (\ref{th1eq2}), we may assume $n$ is large enough so that $\{k\in \bb{N}:k\in n B\}\neq \emptyset$. Note that for any real   $(\alpha_k)_{k=1,\ldots,n}$ and $(\beta_k)_{k=1,\ldots,n}$, $n\in \mathbb{N}$,  we have
	\begin{equation}
	\left|\max_{k=1,\ldots,n} (\alpha_k+\beta_k)-\max_{k=1,\ldots,n} \alpha_k\right|\leq \max_{k=1,\ldots,n}\left|\beta_k\right|.\label{th1eq3}
	\end{equation}
	So
	\begin{align*}
		\limsup_{n\to\infty}\mathbb{P}\left(b_n^{-1}\left| M_{n}({B})-M^{(2)}_{n,m}({B})\right|>{\epsilon}\right)&\leq \limsup_{n\to\infty}\mathbb{P}\left(\max_{k\in n B\cap\mathbb{N}}\left| X_{k,m}^{(1)}\right|>b_n{\epsilon}\right)\\
		&\leq \limsup_{n\to\infty}n\mathbb{P}\left(\left| X_{{1},m}^{(1)}\right|>b_n{\epsilon}\right)=0,
	\end{align*}
where the last limit follows from $b_n \in \mathrm{RV}_\infty ((1-\beta)/\alpha)$ and
	\begin{equation*}
		\mathbb{P}\left(\left| X_{k,m}^{(1)}\right|>x\right)=o\left(e^{-\delta x\log(x)}\right)  ~~\text{as }x\to\infty 
	\end{equation*}
	for some $\delta>0$, a tail property for an infinitely divisible distribution with  a bounded L\'evy measure (Theorem 26.1 in \cite{ken1999levy}).
\end{proof}

\subsection{Infinite variance case}\label{sec:pf inf var}

Recall the measure space $({\mathsf{E}},\cl{E},\mu)$ introduced in Section \ref{subsec1} used to define $\{X_n\}_{n\in\mathbb{N}}$ in $(\ref{eq1})$.
For each fixed $n\in \bb{N}$, let ${\left\{U^{(n)}_j\right\}}_{j\in\mathbb{N}}$ be iid ${\mathsf{E}}$-valued random variables with common law $\mu_n$  determined by
\begin{equation*}
\frac{d\mu_n}{d\mu} (\mathbf{x}) = \frac{\mathbf{1}_{\{\cup_{k=1}^n T^{-k}A\}}(\mathbf{x})}{\mu\left(\cup_{k=1}^n T^{-k}A\right)},~~ \mathbf{x}\in {\mathsf{E}}.
\end{equation*} 
Due to the shift invariance $\mu(T^{-1}\cdot)=\mu(\cdot)$ and $\mu(\cup_{k=1}^n T^{-k}A)=w_{n}$ in view of \eqref{def1}, we have 
\begin{equation}\label{eq:U entrance}
\bb{P}\left(U^{(n)}_j\in T^{-k} A\right)= \mu(T^{-k}A)w_{n}^{-1}=\mu(A)w_{n}^{-1}=w_{n}^{-1},\quad k=1,\ldots,n, \ j\in \bb{N}.
\end{equation}
As before  let ${\{}\varepsilon_j{\}}_{j\in\mathbb{N}}$ be a sequence of iid Rademacher random variables and let ${\{}\Gamma_j{\}}_{j\in\mathbb{N}}$ be a sequence of ordered standard Poisson arrivals on $[0,\infty)$. 
The three sequences mentioned are assumed independent. It follows from Theorem 3.4.3 in \cite{samorodnitsky2016stochastic} that the process $\{X_n\}_{n\in\mathbb{N}}$   admits a   series representation  for each fixed $n\in\mathbb{N}$:
\begin{equation*}
(X_k)_{k=1,\cdots,n}\stackrel{d}{=}\left(\sum_{j=1}^\infty\varepsilon_j\rho^{\leftarrow}(\Gamma_j/(2w_{n}))\mathbf{1}_{\{U^{(n)}_j\in T^{-k}A\}}\right)_{k=1,\cdots,n};\label{series}
\end{equation*}
see also \cite{bai2020functional}. In view of this series representation, we introduce  the processes  
\begin{equation}
\left(S^*_n(t)\right)_{t\in [0,1]}:=\left(\sum_{k=1}^{\lfloor nt\rfloor}\sum_{j=1}^\infty\varepsilon_j\rho^{\leftarrow}(\Gamma_j/(2w_{n}))\mathbf{1}_{\{U^{(n)}_j\in T^{-k}A\}}\right)_{t\in [0,1]} \label{eqsum}
\end{equation}
and 
\begin{equation}
(M^*_n(B))_{B\in \cl{G}([0,1])}:=\left(\max_{k\in nB\cap\mathbb{N}}\sum_{j=1}^\infty\varepsilon_j\rho^{\leftarrow}(\Gamma_j/(2w_{n}))\mathbf{1}_{\{U^{(n)}_j\in T^{-k}A\}}\right)_{B\in \cl{G}([0,1])}.\label{eqmax}
\end{equation}
Now with $\overset{fdd}{=}$ denoting equality in finite-dimensional distributions,    the partial sum process $(S_n(t))_{t\in [0,1]}$  in $(\ref{sum})$ and   the random sup measure $(M_n(B))_{B\in \cl{G}([0,1])}$ in $(\ref{max})$  satisfy for each  fixed $n\in \bb{N}$,
\begin{equation}\label{eq:equiv rep}
\begin{pmatrix}
\left(S_n(t)\right)_{t\in [0,1]}\\
(M_n(B))_{B\in \cl{G}([0,1])}
\end{pmatrix}\overset{fdd}{=}\begin{pmatrix}
\left(S^*_n(t)\right)_{t\in [0,1]}\\
(M^*_n(B))_{B\in \cl{G}([0,1])}
\end{pmatrix}.
\end{equation}

Now we proceed to introducing a Poissonization construction key to the proof.
For each $j,n\in\mathbb{N}$, define  the set of scaled entrance times with a random starting point $ U_j^{(n)}$ as
\begin{equation}\label{eq:R tilde j,n}
\wt{R}_{j,n} =\frac{1}{n}\left\{k=1,\cdots,n: \  U_j^{(n)}\in T^{-k}A\right\}.
\end{equation}
Due to the construction of $U_j^{(n)}$, $\wt{R}_{j,n}$ is nonempty, say with (random) cardinality  $Q_{j,n}\in \bb{N}$ , and hence we can write
\[
\wt{R}_{j,n}=V_{j,n}+\left\{\tau_{j,n}(0),\tau_{j,n}(1),\ldots, \tau_{j,n}(Q_{j,n} -1 ) \right\}\subset  n^{-1}\{1,\cdots,n\} ,
\]
where
\begin{equation}
V_{j,n}:=\frac{1}{n}\min\left\{k=1,\cdots,n:\ U^{(n)}_j\in T^{-k}A\right\},\label{vjn}
\end{equation}
is the scaled first entrance time and
$\tau_{j,n}(0):=0<\tau_{j,n}(1)<\ldots <\tau_{j,n}(Q_{j,n} -1 )$ 
are the scaled successive entrance times  relative to $V_{j,n}$.  In view of the Markov chain construction in Section \ref{subsec1} (cf.\  \cite{samorodnitsky2019extremal} and \cite{bai2022tail}), ${\{}n \tau_{j,n}(i){\}}_{i}$ forms an $\mathbb{N}$-valued renewal process (i.e.,   random walk with $\mathbb{N}$-valued iid\ steps) starting at $i=0$ and stopped at $i=Q_{j,n} -1 $, {and $\lim_{n\to\infty} Q_{j,n}=\infty$}. The inter-arrival times of ${\{}n \tau_{j,n}(i){\}}_{i}$, before stopping, have common probability mass function $P_{0} (\varphi_{A}=k)$ for $k= 1,2,\cdots$. For convenience, we shall extend the definition of $\tau_{j,n}(i)$ for $i\ge Q_{j,n}$, by adding with iid inter-arrivals,  while maintaining independence between different $j$'s for $\{\tau_{j,n}(i)\}_{i\in \mathbb{N}_0}$.

For each $j$ and $n$, let $\{N_{j,n}(t)\}_{t\geq 0}$ be a Poisson process with intensity
\begin{equation}\label{eq:gamma_n}
\gamma_n:=nw_n^{-1}/\Gamma(2-\beta)\in \text{RV}_\infty(\beta),
\end{equation} and independent of everything else. Define the non-decreasing processes $\sigma_{j,n}$ by 
\begin{equation}
\{\sigma_{j,n}(t):t\geq 0\}:=\{\tau_{j,n}(N_{j,n}(t)):t\geq 0\}.\label{sjn}
\end{equation}
Then each $\sigma_{j,n}$ is a non-decreasing compound Poisson process and hence a subordinator. Let $L_{j,n}$ and $R_{j,n}$ be the local time and range of $\sigma_{j,n}$ defined in (\ref{eqlocal}) and (\ref{eqrange}), respectively.  Note that   $\wt{R}_{j,n}=(V_{j,n}+R_{j,n})\cap[0,1]$. For each    subset  $S\subset\bb{N}$, define  the intersection of shifted random sets
\begin{equation}\label{eq:I_S,n}
I_{S,n}:=\begin{cases}
\bigcap_{j\in S}\wt{R}_{j,n}=\bigcap_{j\in S} (V_{j,n}+R_{j,n})\cap[0,1],\quad & \emptyset\neq S\subset \bb{N};\\
 n^{-1}\{1,\cdots,n\} , \quad  &   S=\emptyset. 
\end{cases}
\end{equation} 
Now set a level $\ell\in \bb{N}$, and introduce 
\begin{equation}
S^*_{n,\ell}(t):=\sum_{j=1}^\ell\varepsilon_j\rho^{\leftarrow}(\Gamma_j/(2w_{n}))L_{j,n}\left((t-V_{j,n})_+\right),~~t\in[0,1]\label{intsum}
\end{equation}
and
\begin{equation}
M^*_{n,\ell}(B):=\max_{S\subset\{1,\cdots,\ell\}}\mathbf{1}_{\{I_{S,n}\cap B\neq\emptyset\}}\sum_{j\in S}\mathbf{1}_{\{\varepsilon_j=1\}}\rho^{\leftarrow}(\Gamma_j/(2w_{n})),~~B\in\cl{G}([0,1]),\label{intmax}
\end{equation}
where a summation over an empty index set is understood as zero here and below.  We mention that \eqref{intsum} and \eqref{intmax} are respectively approximations of \eqref{eqsum}  and \eqref{eqmax}, as will be clarified in Lemmas \ref{lemsum} and \ref{lemmax} below. 
Define also the truncated version of $S(t)$ in (\ref{limsum}) by
\begin{equation}
S_{\ell}(t):=(2{ C_\alpha})^{1/\alpha}\sum_{j=1}^\ell\varepsilon_j\Gamma_j^{-1/\alpha}L_{j}\left((t-V_{j})_+\right),~~t\in[0,1].\label{trusum}
\end{equation}
and the truncated version of $M(B)$ in (\ref{maxidx}) as
\begin{equation}
	M_{\ell}(B):=\max_{S\subset\{1,\cdots,\ell\}  }\mathbf{1}_{\{I_{S}\cap B\neq\emptyset\}}\sum_{j\in S}\Gamma_j^{-1/\alpha},~~B\in \cl{G}([0,1]).\label{trumax}
\end{equation}

\begin{proposition}\label{thmapprox}
	Fix an integer $\ell>0$. Let $S^*_{n,\ell}$ and $M^*_{n,\ell}$ be defined by $(\ref{intsum})$ and $(\ref{intmax})$, respectively. Let $S_{\ell}$ and $M_{\ell}$ be defined by $(\ref{trusum})$ and $(\ref{trumax})$, respectively. Then
	\begin{equation}
	\frac{1}{\rho^{\leftarrow}(w_{n}^{-1})}\begin{pmatrix}
	S^*_{n,\ell}\\
	M^*_{n,\ell}
	\end{pmatrix}\stackrel{fdd}{\to}\begin{pmatrix}
	C_\alpha^{-1/\alpha}S_{\ell}\\
	M_{\ell}
	\end{pmatrix}~~\text{as }n\to\infty\label{thmappeq1}
	\end{equation}
	 in $T_0:=[0,1]\times \cl{I}$, where $\mathcal{I}$ is the collection of all nonempty open sub-intervals of $[0,1]$. 
\end{proposition}

\begin{proof}[Proof of Proposition~{\upshape\ref{thmapprox}}]
	Since $(P_{0} (\varphi_{A}>n))_n\in \text{RV}_\infty(-\beta)$, write $P_{0} (\varphi_{A}>n)=: n^{-\beta} f(n) $ with $ f(n) $ slowly varying.  Recall each $\tau_{j,n}$ is an increasing random walk with infinite-mean regularly varying step size, and  $\{\sigma_j\}$, $\{L_j\}$ and $\{V_j\}$ are  as described before \eqref{limsum}. A classical result for convergence to stable law (see e.g. Section 1.a. in \cite{chow1978sum}) entails that
	\begin{equation*}
	\frac{1}{\gamma_n^{1/\beta} {\widetilde{f^{-1/\beta}}} (\gamma_n^{1/\beta})}n\tau_{j,n}(\lfloor\gamma_n\rfloor)\Rightarrow {\left(\Gamma(1-\beta)\right)}^{1/\beta} \sigma_j(1)~~\text{as }n\to\infty,
	\end{equation*}
	where  $\widetilde{f^{-1/\beta}}$  is the slowly varying function conjugate to $f^{-1/\beta}$  satisfying 
\begin{equation}\label{eq:slowly var conj}
 \lim_{n\to\infty}f^{-1/\beta}(n)\widetilde{f^{-1/\beta}}(nf^{-1/\beta}(n))=1 
\end{equation}	
	(see e.g. Theorem 1.5.13 in \cite{bingham1989regular}) and $\gamma_n$ is as in \eqref{eq:gamma_n}. 
	 Note that by the relation in \eqref{ass1}, we have as $n\rightarrow\infty$,
	\begin{equation*}
	\gamma_n^{1/\beta}\sim nf^{-1/\beta}(n)\left(\Gamma(1-\beta)\right)^{-1/\beta}.
	\end{equation*}
 This together with \eqref{eq:slowly var conj} 
	 implies $\lim_{n\to\infty}\gamma_n^{-1/\beta}\{\widetilde{f^{-1/\beta}}(\gamma_n^{1/\beta})\}^{-1}n={\left(\Gamma(1-\beta)\right)}^{1/\beta}$ and hence
	\begin{equation*}
	\tau_{j,n}(\lfloor\gamma_n\rfloor)\Rightarrow\sigma_j(1)~~\text{as }n\to\infty.
	\end{equation*}
	Since $N_{j,n}(1)/\gamma_n\stackrel{\mathbb{P}}{\to}1$, $N_{j,n}(1)$ is independent of $\tau_{j,n}$, and 
	$(\gamma_n)\in  \text{RV}_\infty(\beta)$, by a standard argument replacing the deterministic  $\floor{\gamma_n}$ with  the random   $N_{j,n}(1)$, we have	
	\begin{equation*}
	\sigma_{j,n}(1)\Rightarrow\sigma_j(1)~~\text{as }n\to\infty,\label{thmapproxeq2}
	\end{equation*}
	where $\sigma_{j,n}$ is defined by \eqref{sjn}. In addition, 
	as a direct result from the proof of Theorem 5.4  in \cite{samorodnitsky2019extremal}, we have
	\begin{equation*}
	V_{j,n}\stackrel{d}{\to}V_j~~\text{as }n\to\infty,\label{thmapproxeq3}
	\end{equation*}
	where $V_{j,n}$ is as  in \eqref{vjn}.  Now applying  Proposition \ref{thmsub} with independence between $j$'s, we have
	\begin{equation}
	\begin{pmatrix}
	\left(L_{j,n}\left((x-V_{j,n})_+\right)\right)_{x\in[0,1]}\\
	(V_{j,n}+R_{j,n})\cap [0,1]
	\end{pmatrix}_{j=1,\cdots,\ell}\Rightarrow\begin{pmatrix}
	\left(L_j\left((x-V_j)_+\right)\right)_{x\in[0,1]}\\
	(V_j+R_j)\cap[0,1]
	\end{pmatrix}_{j=1,\cdots,\ell}\label{thmapproxeq4}
	\end{equation} 
	weakly in  $ D[0,1]^{\ell}\times\mathfrak{F}([0,1])^{\ell}$  as $n\to\infty$.  By Theorem 5.4  in \cite{samorodnitsky2019extremal}, we have the marginal convergence for each $S\subset \{1,\ldots,\ell\}$:
	\begin{equation}
	I_{S,n}\Rightarrow I_S\label{thmapproxeq5}
	\end{equation}
	in  $\mathfrak{F}([0,1])$  as $n\to\infty$, where $I_{S,n}$ is as in \eqref{eq:I_S,n} and $I_S$ is as in \eqref{eq:I_S}. By  
	Lemma \ref{lemconv} below (an extension of \cite[Theorem 2.1]{samorodnitsky2019extremal}), the relations \eqref{thmapproxeq4} and \eqref{thmapproxeq5} imply the joint convergence
	\begin{equation}
	\begin{pmatrix}
	\left(L_{j,n}\left((x-V_{j,n})_+\right)\right)_{x\in[0,1]}\\
	I_{S,n}
	\end{pmatrix}_{\substack{j=1,\cdots,\ell\\S\subset\{1,\cdots,\ell\}}}\Rightarrow\begin{pmatrix}
	\left(L_j((x-V_j)_+)\right)_{x\in[0,1]}\\
	I_S
	\end{pmatrix}_{\substack{j=1,\cdots,\ell\\S\subset\{1,\cdots,\ell\}}}\label{thmapproxeq6}
	\end{equation}
	weakly in $D[0,1]^{\ell}\times\mathfrak{F}([0,1])^{2^\ell}$ as $n\to\infty$. Now \eqref{thmappeq1} is a consequence of \eqref{thmapproxeq6} and the facts that 
	\begin{equation*}
	\frac{\rho^{\leftarrow}(\Gamma_j/(2w_{n}))}{\rho^{\leftarrow}\left(w_{n}^{-1}\right)}\to (\Gamma_j/2)^{-1/\alpha},~~j=1,\cdots,\ell,
	\end{equation*}
due to the regular variation mentioned below \eqref{inv} 	and
	\begin{equation}
	(M_{\ell}(B))_{B\in \cl{G}([0,1])}\eqfdd\left(\max_{S\subset\{1,\cdots,\ell\}}\mathbf{1}_{\{I_{S}\cap B\neq\emptyset\}}\sum_{j\in S}\mathbf{1}_{\{\varepsilon_j=1\}}(\Gamma_j/2)^{-1/\alpha}\right)_{B\in \cl{G}([0,1])}.\label{thmapproxeq1}
	\end{equation}
	Here  \eqref{thmapproxeq1} holds because the thinned Poisson   process ${\{}\Gamma_j /2{\}}_{j\in\mathbb{N},\varepsilon_j=1}$ equals in law to ${\{}\Gamma_j{\}}_{j\in\mathbb{N}}$. 	
\end{proof}

\begin{lemma}\label{lemconv} 
	Let $\{A_k\}_{k=1,\cdots,m}$ and $\{A_k(n)\}_{n\in\mathbb{N},k=1,\cdots,m}$ be random closed
	sets in $\mathfrak{F} = \mathfrak{F}(\mathbb{R}^d )$, $m\in \bb{N}$.  Let ${x}(n)$ and ${x}$ be random  elements  in ${\mathsf{E}}$, where ${\mathsf{E}}$ is a separable metric space. Suppose the following joint weak convergence holds:
	\begin{equation}
	\begin{pmatrix}
	{x}(n)\\A_k(n)
	\end{pmatrix}_{k=1,\cdots,m}\Rightarrow\begin{pmatrix}
	{x}\\A_k\end{pmatrix}_{k=1,\cdots,m}\label{wkconv}
	\end{equation}
	in ${\mathsf{E}}\times\mathfrak{F}^m$ as $n\to\infty$. For any nonempty $I\subset \{1,\cdots,m\}$, set the intersections of random closed sets   $A_I (n) = \bigcap_{k\in I}	A_k(n)$ and $A_I = \bigcap_{k\in I}A_k$.   We also set  $A_\emptyset(n),A_\emptyset\in\mathfrak{F}$  to be nonrandom such that $A_k(n)\subset A_\emptyset(n)$ and $A_k\subset A_\emptyset$ for all $k$. Assume in addition the marginal convergence of intersections:
	\begin{equation*}
	A_I (n)\Rightarrow A_I ~as~ n \to\infty, \text{ for each}~I \subset \{1,\cdots,m\}.
	\end{equation*}
	Then
	\begin{equation}
	\begin{pmatrix}
	{x}(n)\\
	(A_I (n))_{	I \subset \{1,\cdots,m\}}
	\end{pmatrix} \Rightarrow\begin{pmatrix}
	{x}\\
	(A_I )_{I \subset \{1,\cdots,m\}}
	\end{pmatrix},\label{eq:joint goal}
	\end{equation}
	in  ${\mathsf{E}}\times \mathfrak{F}^{2^m}$.
\end{lemma}
\begin{proof}[Proof of Lemma~{\upshape\ref{lemconv}}]
The only essential difference between \cite[Theorem 2.1(b)]{samorodnitsky2019extremal}   and the current lemma is the inclusion of the components  ${x}(n)$ and ${x}$.
 Note that the product space   ${\mathsf{E}}\times\mathfrak{F}^m$  is a separable metric space too since both ${\mathsf{E}}$ and  $\mathfrak{F}$  are. Then  the Skorokhod’s representation theorem applies to   replace (\ref{wkconv})  with an almost sure convergence.  It remains to follow the proof of  \cite[Theorem 2.1(b)]{samorodnitsky2019extremal} to establish a convergence-in-probability version of \eqref{eq:joint goal}.  
\end{proof}

With Proposition \ref{thmapprox}, we need to establish certain triangular approximation results in order to reach the conclusion  of Theorem \ref{thmmain}. First is an approximation result for the partial sum process.  Second is for the partial maximum process. 
\begin{lemma}\label{lemsum}
	For any $t\in[0,1]$ and for any $\epsilon>0$,
	\begin{equation}
	\lim_{\ell\to\infty}\limsup_{n\to\infty}\mathbb{P}\left(\left| c_n^{-1}S^*_{n}(t)-\{\rho^{\leftarrow}(w_{n}^{-1})\}^{-1}S^*_{n,\ell}(t)\right|>\epsilon\right)=0,\label{propeq1}
	\end{equation}
	where $S^*_{n,\ell}$ is as in \eqref{intsum}, $S_n^*$ is as in \eqref{eqsum}, and $c_n=\rho^{\leftarrow}(w_n^{-1})nw_n^{-1}/\Gamma(2-\beta)$.
\end{lemma}

\begin{proof}[Proof of Lemma~{\upshape\ref{lemsum}}]
	Define the truncated version of $S^*_n(t)$ in \eqref{eqsum} as
	\begin{equation*}
	S^{*'}_{n,\ell}(t):=\sum_{k=1}^{\lfloor nt\rfloor}\sum_{j=1}^\ell\varepsilon_j\rho^{\leftarrow}(\Gamma_j/(2w_{n}))\mathbf{1}_{\{U^{(n)}_j\in T^{-k}A\}}.
	\end{equation*} 
By triangular inequalities,	(\ref{propeq1}) will follow once we show that
	\begin{equation}
	\lim_{\ell\to\infty}\limsup_{n{\to\infty}}\mathbb{P}\left(c_n^{-1}\left| S^*_n(t)-S^{*'}_{n,\ell}(t)\right|>\epsilon\right)=0, \label{lemsum1}
	\end{equation}
	and  for any fixed $\ell\in \bb{N}$, we have 
	\begin{equation}
	 \lim_{n{\to\infty}} \mathbb{P}\left(\left| c_n^{-1}S^{*'}_{n,\ell}(t)-\{\rho^{\leftarrow}(w_{n}^{-1})\}^{-1}S^{*}_{n,\ell}(t)\right|>\epsilon\right)=0. \label{lemsum2}
	\end{equation}
	
The relation \eqref{lemsum1}  is a special case of the relation (66) in \cite{bai2020functional}. A similar argument  will also be used   in the proof of \eqref{propmax1} below. So we omit the details here.

To prove \eqref{lemsum2}, using the relation $c_n=\rho^{\leftarrow}(w_n^{-1})\gamma_n$, where $\gamma_n$ is as in \eqref{eq:gamma_n}, it is enough to show that for each $j=1,\ldots,\ell$,
	\begin{equation*}
	\lim_{n\to\infty}\mathbb{P}\left(\frac{\rho^{\leftarrow}(\Gamma_j/(2w_{n}))}{\rho^{\leftarrow}(1/w_{n})}\left|\gamma_n^{-1}\sum_{k=1}^{\floor {nt}}\mathbf{1}_{\{U^{(n)}_j\in T^{-k}A\}}-L_{j,n}\left((t-V_{j,n})_+\right)\right|>\epsilon\right)=0.
	\end{equation*}
	Note that  $\rho^{\leftarrow}(\Gamma_j/(2w_{n}))/\rho^{\leftarrow}(1/w_{n})\to (\Gamma_j/2)^{-1/\alpha}$ as $n\to\infty$ by the regular variation of $\rho^{\leftarrow}$. So it is sufficient to show that
	\begin{equation*}
	\lim_{n\to\infty}\mathbb{P}\left(\left|\gamma_n^{-1}\sum_{k=1}^{\floor {nt}}\mathbf{1}_{\{U^{(n)}_j\in T^{-k}A\}}-L_{j,n}\left((t-V_{j,n})_+\right)\right|>\epsilon\right)=0.\label{propeq5}
	\end{equation*}
	Denote\[Q_{j,n}(t):=\sum_{k=1}^{\floor{ nt}}\mathbf{1}_{\{U^{(n)}_j\in T^{-k}A\}}.\] Observe that   $Q_{j,n}(t)\geq 1$ 
 if and only if $V_{j,n}\leq \floor{nt}/n$ (or equivalently, $V_{j,n}\le t $), under which
	\begin{equation*}
	\tau_{j,n}(Q_{j,n}(t)-1)\leq t-V_{j,n}~\text{ and }~\tau_{j,n}(Q_{j,n}(t))>t-V_{j,n}.
	\end{equation*}
So 	 we have
\begin{align*} 
L_{j,n}\left((t-V_{j,n})_+\right)&=  \inf\{s\ge 0: \tau_{j,n}(N_{j,n}(s))> (t-V_{j,n})_+\}\\
&=\inf\{s\ge 0:\ N_{j,n}(s)=Q_{j,n}(t)\vee 1\} .
\end{align*}
Let  ${\{}T_{j,n}{(i)}{\}}_{i\in\mathbb{N}}$  denote the inter-arrival times of the Poisson process $N_{j,n}$. Then  ${\{}T_{j,n}{(i)}{\}}_{i\in\mathbb{N}}$  are iid exponential random variables with mean $\gamma_n^{-1}$  and
	\begin{equation*}
	L_{j,n}\left((t-V_{j,n})_+\right)=\sum_{i=1}^{Q_{j,n}(t) \vee 1 }T_{j,n}(i).
	\end{equation*}
	Then 
	\begin{align*}
	&\mathbb{P}\left(\left|\gamma_n^{-1}\sum_{k=1}^{\floor{nt}}\mathbf{1}_{\{U^{(n)}_j\in T^{-k}A\}}-L_{j,n}\left((t-V_{j,n})_+\right)\right|>\epsilon\right)\\
	\leq  &\mathbb{P}\left(\left|\frac{Q_{j,n}(t)}{\gamma_n}-\sum_{i=1}^{Q_{j,n}(t)}T_{j,n}{(i)}\right|>\epsilon\right)+ \mathbb{P}\left(T_{j,n}{(1)}>\epsilon\right) .
	\end{align*}	
    Note that $\lim_n\mathbb{P}(T_{j,n}{(1)}>\epsilon)=0$ since $\gamma_n\rightarrow \infty$ as $n\rightarrow\infty$. Furthermore, by Chebyshev's inequality, independence, and recalling $(w_n)\in \mathrm{RV}_\infty(1-\beta)$, $(\gamma_n)\in \mathrm{RV}_\infty( \beta)$ and \eqref{eq:U entrance}, we have
	\begin{equation*}
	\mathbb{P}\left(\left|\frac{Q_{j,n}(t)}{\gamma_n}-\sum_{i=1}^{Q_{j,n}(t)}T_{j,n}{(i)}\right|>\epsilon\right)\le \frac{\mathbb{E}(Q_{j,n}(t))\mathrm{Var}(T_{j,n}{(1)})}{\epsilon^2}  \leq   \frac{(n/w_{n})\gamma_n^{-2}}{\epsilon^2},
	\end{equation*}
which converges to $0$	as $n\rightarrow\infty$ since the last expression belongs to $\mathrm{RV}_\infty(-\beta)$.
\end{proof}

\begin{lemma}\label{lemmax}
	For any nonempty open interval $B\subset [0,1]$ and for any $\epsilon>0$,
	\begin{equation}
	\lim_{\ell\to\infty}\limsup_{n\to\infty}\mathbb{P}\left(b_n^{-1}\left|M^*_{n}(B)-M^*_{n,\ell}(B)\right|>\epsilon\right)=0,\label{propeq2}
	\end{equation}
where $b_n=\rho^{\leftarrow}(w_n^{-1})$, $M^*_{n}$ is as in \eqref{eqmax} and $M^*_{n,\ell}$ is as in \eqref{intmax}.
\end{lemma}

\begin{proof}[Proof of Lemma~{\upshape\ref{lemmax}}]
	Define
	\begin{equation}\label{eq:M prime}
	M^{*'}_{n,\ell}(B):=\max_{k\in nB\cap\mathbb{N}}\sum_{j=1}^\ell\varepsilon_j\rho^{\leftarrow}(\Gamma_j/(2w_{n}))\mathbf{1}_{\{U^{(n)}_j\in T^{-k}A\}}
	\end{equation}
	where $\max_\emptyset=-\infty$. The conclusion \eqref{propeq2} follows once we show that
	\begin{equation}
	\lim_{\ell\to\infty}\limsup_{n{\to\infty}}\mathbb{P}\left(b_n^{-1}\left| M^*_{n}(B)-M^{*'}_{n,\ell}(B)\right|>\epsilon\right)=0, \label{propmax1}
	\end{equation}
	and for any fixed $\ell\in \bb{N}$ that
	\begin{equation}
 \lim_{n{\to\infty}}\mathbb{P}\left( b_n^{-1}\left|M^{*'}_{n,\ell}(B)-M^*_{n,\ell}(B)\right|>\epsilon\right)=0.\label{propmax2}
	\end{equation}
	
	To show \eqref{propmax1}, by \eqref{th1eq3}, it is enough to show that
	\begin{equation}
	\lim_{\ell\to\infty}\limsup_{n{\to\infty}}\mathbb{P}\left(b_n^{-1}\max_{k\in nB\cap\mathbb{N}}\left|\sum_{j=\ell+1}^\infty\varepsilon_j\rho^{\leftarrow}(\Gamma_j/(2w_{n}))\mathbf{1}_{\{U^{(n)}_j\in T^{-k}A\}}\right|>\epsilon\right)=0.\label{propmax4}
	\end{equation}
By a union bound and a Markov inequality,	the  probability  in \eqref{propmax4} is bounded above by
	\begin{align*}
	&n\mathbb{P}\left(b_n^{-1}\left|\sum_{j=\ell+1}^\infty\varepsilon_j\rho^{\leftarrow}(\Gamma_j/(2w_{n}))\mathbf{1}_{\{U^{(n)}_j\in T^{-k}A\}}\right|>\epsilon\right)\\
	\leq& n\epsilon^{-r}\mathbb{E}\left(b_n^{-1}\left|\sum_{j=\ell+1}^\infty\varepsilon_j\rho^{\leftarrow}(\Gamma_j/(2w_{n}))\mathbf{1}_{\{U^{(n)}_j\in T^{-k}A\}}\right| \right)^r\\
	\leq& n\epsilon^{-r}C\left( \sum_{j=\ell+1}^\infty\mathbb{E} \left[\frac{\rho^{\leftarrow}(\Gamma_j/(2w_{n}))^2}{\rho^{\leftarrow}(w_{n}^{-1})^2} \right]\mathbb{P}(  U^{(n)}_j\in T^{-k}A )\right)^{r/2}
	\end{align*}
	for some $r>0$ to be specified later, where $C$ is a constant  depending only  on $r$ (throughout $C$ denotes a generic positive constant that may change each time), and the last {in}equality follows from independence and the  Khinchine inequality for  Rademacher random variables. Using the inequality (82) in the proof of equation (66) in \cite{bai2020functional},  for $j$ large enough,
	\begin{equation*}
	\mathbb{E} \left[\frac{\rho^{\leftarrow}(\Gamma_j/(2w_{n}))^2}{\rho^{\leftarrow}(w_{n}^{-1})^2} \right]\leq C\mathbb{E}\left((\Gamma_j/2)^{-1/\alpha_0}+(\Gamma_j/2)^{-(1/\alpha)-\delta}\right)^2\leq Cj^{-2\gamma},
	\end{equation*} 	
	for some positive constants $C$ and $\delta$, where $\alpha_0\in (0,2)$ is as in \eqref{eqasssum} and $\gamma:=\min\{1/\alpha_0,1/\alpha+\delta\}$. Using also $\mathbb{P}(  U^{(n)}_j\in T^{-k}A )=w_{n}^{-1}$, we have
	\begin{align}
	&\mathbb{P}\left(b_n^{-1}\max_{k\in nB\cap\mathbb{N}}\left|\sum_{j=\ell+1}^\infty\varepsilon_j\rho^{\leftarrow}(\Gamma_j/(2w_{n}))\mathbf{1}_{\{U^{(n)}_j\in T^{-k}A\}}\right|>\epsilon\right)\notag\\
	\leq &n\epsilon^{-r}Cw_{n}^{-r/2}\sum_{j=\ell+1}^\infty j^{-2\gamma}.\label{eq:max tri approx}
	\end{align}
Recall  $(w_n)\in \mathrm{RV}_\infty(1-\beta)$. With the choice $r > 2/(1-\beta)$, we ensure
	\begin{equation}
	\lim_{n\to\infty}n w_{n}^{-r/2}=0,\label{expchoi}
	\end{equation}
	and hence (\ref{propmax4}) follows from \eqref{eq:max tri approx} and  \eqref{expchoi} since ${2\gamma>1}$.

Now we show \eqref{propmax2}. 
	Recall $\wt{R}_{j,n}$ in \eqref{eq:R tilde j,n} and $I_{S,n}$ in \eqref{eq:I_S,n}. Writing $S^c:=\{1,\cdots,\ell\}\setminus S$, introduce
	\begin{equation*}
	I^*_{S,n}:=I_{S,n}\cap\left(\bigcap_{j\in S^c}  \wt{R}_{j,n}^c\right).
	\end{equation*}
	So each (rescaled) time point in $I^*_{S,n}$ is exactly contained by those $\wt{R}_{j,n}$ with $j\in S$.
	    Define the event
%
 \begin{equation*}
	A_n(B)=\bigcup_{S\subset\{1,\cdots,\ell\}}\left(\{I_{S,n}\cap B\neq\emptyset\}\cap\{I^*_{S,n}\cap B=\emptyset\}\right).
	\end{equation*}
In view of Lemma 5.5 in \cite{samorodnitsky2019extremal}, 
we have  
\begin{equation*}\label{eq:M double prime = M prime}
M^{*'}_{n,\ell}(B)=M^*_{n,\ell}(B)  \text{ on }A_n(B)^c  \text{ with } \lim_{n\to\infty}\mathbb{P}(A_n(B))=0.
\end{equation*}
\eqref{propmax2} follows immediately.
\end{proof}

\begin{proof}[Proof of Theorem~{\upshape\ref{thmmain}}]
	The tightness holds on the product space if it holds on each marginal space. The tightness of $c_n^{-1}S_n$ on $J_1$ topology of $D[0,1]$ is proved in \cite{owada2015functional}; see also \cite{bai2020functional}. $b_n^{-1}M_n$ is automatically tight since the space $SM[0,1]$ is compact. In view of Proposition \ref{thmfdd} and Example \ref{eg:fdd}, we only need to show the convergence of  finite-dimensional distributions  in $T_0=[0,1]\times \mathcal{I} $ (recall $\cl{I}$ is the collection of all nonempty open subintervals of $[0,1]$)   using a triangular approximation argument  (Theorem 3.2 in \cite{billingsley1999convergence}).
{Note that as $\ell\rightarrow\infty$, $S_\ell(t)\rightarrow S(t)$ and $M_\ell(B)\rightarrow M(B)$ almost surely for any $t\in [0,1]$ and $B\in \mathcal{G}([0,1])$,}
where $S_{\ell}$ and $M_{\ell}$ defined in (\ref{trusum}) and (\ref{trumax}) are the truncated versions of $S$ and $M$, respectively. The triangular approximation argument  is then completed  by   \eqref{eq:equiv rep}, Proposition \ref{thmapprox},  Lemmas \ref{lemsum} and \ref{lemmax}.\\
\end{proof}	

\subsection{Dependence of limits in the infinite variance case}\label{sec:dep}
\begin{proof}[Proof of Proposition \ref{Pro:dep}]	
	It is enough to show that $S(1)$ and $M([0,1])$ are dependent. We shall do so by   showing the so-called tail dependence coefficient  
\begin{equation}\label{eq:tail dep coef}
\lim_{x\rightarrow\infty}\bb{P}(|S(1)|>x \ | \ M([0,1])>x)\neq 0.
\end{equation}

	Denote  \[{Z}_j:={(2{ C_\alpha})}^{1/\alpha}\varepsilon_jL_{j}(1-V_{j}) ,\quad  j\in \bb{N},
	\] which admits finite moments of any order since the marginal Mittag-Leffler law of $L_j(1)$ does.   
Then $S(1)=\sum_{j=1}^\infty {Z}_j\Gamma_j^{-1/\alpha}$.  By   \eqref{maxidx}, the ordering $\Gamma_1<\Gamma_2<\ldots$,  and the fact that $I_S\neq \emptyset$ almost surely only for any $|S|\le \ell_\beta$, one has  almost surely 
\[
M([0,1]) =\sup_{S\subset \bb{N}, 1\le |S|\le \ell_\beta} \left(\sum_{j\in S}\Gamma_j^{-1/\alpha}\right)=\sum_{j=1}^{\ell_\beta}\Gamma_j^{-1/\alpha}.
\]    We claim that, as $x\to\infty$, we have marginally 
\begin{equation}\label{eq:tail S(1)}
\mathbb{P}( |S(1)|>x)\sim \mathbb{P}( |{Z}_1|\Gamma_1^{-1/\alpha}>x)\sim x^{-\alpha} \mathbb{E}[|{Z}_1|^\alpha].
\end{equation}

To show the   relations above, applying the orthogonality $\mathbb{E}({Z}_i{Z}_j)=0$ for $i\neq j$ and independence,  we have for all $\ell$ large enough,  
\begin{equation*}
\mathbb{E}\left(\sum_{j=\ell}^\infty {Z}_j\Gamma_j^{-1/\alpha}\right)^2=\sum_{j=\ell}^\infty \mathbb{E}\left({Z}_j^2\Gamma_j^{-2/\alpha}\right)\le C \mathbb{E}\left({Z}_1^2\right)\sum_{j=\ell}^\infty j^{-2/\alpha}<\infty,
\end{equation*}
where $C$ is a generic positive constant, and we have applied  a bound for negative moments of $\Gamma_j$ (e.g, \cite[Equation (3.2)]{samorodnitsky1989asymptotic}) in the last step. Hence, a Markov inequality yields
\begin{equation}\label{eq:trunc tail}
\mathbb{P}\left(\left|\sum_{j=\ell}^\infty {Z}_j\Gamma_j^{-1/\alpha}\right|>x\right)\leq Cx^{-2}.
\end{equation}
Furthermore, for any fixed $j\in \bb{N}$,   by Breiman's lemma and an elementary calculation based on the gamma distribution of $\Gamma_j$, we have  as $x\rightarrow\infty$,
\begin{equation}\label{eq:AGamma tail}
\mathbb{P}(|{Z}_j|\Gamma_j^{-1/\alpha}>x)\sim \left(\bb{E} |{Z}_j|^{j\alpha}\right)  \bb{P}(\Gamma_j^{-1/\alpha}>x)\sim\frac{ \bb{E} |{Z}_j|^{j\alpha} }{ j!  }x^{-j\alpha}.
\end{equation}
Combining \eqref{eq:trunc tail} and \eqref{eq:AGamma tail} with  Lemma 4.2.4 of \cite{samorodnitsky2016stochastic}, the   relations in \eqref{eq:tail S(1)} follow.

On the other hand,   it  follows from  Proposition 3.3 in \cite{samorodnitsky2019extremal}   that as $x\to\infty$
	\begin{equation*}  
	\mathbb{P}(M([0,1])>x)\sim \mathbb{P}(\Gamma_1^{-1/\alpha}>x)\sim x^{-\alpha}.
	\end{equation*}
	
Next, we consider the joint tail probability.  For any $\epsilon\in (0,1)$, applying  union bounds, triangular inequalities and \eqref{eq:AGamma tail}, we have 
\begin{align*}
&\mathbb{P}\left( |S(1)|>x,M([0,1])>x\right)\\ \le &\mathbb{P}\left( |{Z}_1|\Gamma_1^{-1/\alpha}>(1-\epsilon)x,\ \Gamma_1^{-1/\alpha}>(1-\epsilon)x\right)+\bb{P}\left(\left|\sum_{j=2}^\infty {Z}_j\Gamma_j^{-1/\alpha}\right|>\epsilon x\right)+ \bb{P}\left(\sum_{j=2}^{\ell_\beta}\Gamma_{j}^{-1/\alpha}>\epsilon x\right)\\
=&\mathbb{P}\left( |{Z}_1|\Gamma_1^{-1/\alpha}>(1-\epsilon)x,\ \Gamma_1^{-1/\alpha}>(1-\epsilon)x\right)+ o(x^{-\alpha}) \quad \text{ as }x\rightarrow\infty,
\end{align*}
where the last relation is due to  \eqref{eq:trunc tail} and \eqref{eq:AGamma tail}.
Note that
\begin{align*}
&\mathbb{P}( |{Z}_1|\Gamma_1^{-1/\alpha}>x,\Gamma_1^{-1/\alpha}>x)\\
=& \mathbb{P}(  \mathbf{1}_{\{|{Z}_1|\leq 1\}} |{Z}_1|\Gamma_1^{-1/\alpha}>x)+\mathbb{P}( |{Z}_1|> 1)\mathbb{P}( \Gamma_1^{-1/\alpha}>x)\\
\sim&\left\{ \mathbb{E}[|{Z}_1|^\alpha\mathbf{1}_{\{|{Z}_1|\leq 1\}}]+\mathbb{P}(|{Z}_1|> 1)\right\}x^{-\alpha} 
\end{align*}
as $x\rightarrow\infty$, where   we have applied	Breiman's lemma in the last relation.
Combining the relations above we have
\begin{equation*}
\limsup_{x\to\infty}\frac{\mathbb{P}( |S(1)|>x,M([0,1])>x)}{\mathbb{P}( |{Z}_1|\Gamma_1^{-1/\alpha}>x,\Gamma_1^{-1/\alpha}>x)}\le(1-\epsilon)^{-\alpha}.
\end{equation*}
Similarly, a lower bound can be obtained by noting that for any $\epsilon\in (0,1)$, 
\begin{align*}
&\mathbb{P}\left( |S(1)|>x,M([0,1])>x\right)\ge \\ &\mathbb{P}\left( |{Z}_1|\Gamma_1^{-1/\alpha}>(1+\epsilon)x,\Gamma_1^{-1/\alpha}>(1+\epsilon)x\right)-\bb{P}\left(\left|\sum_{j=2}^\infty {Z}_j\Gamma_j^{-1/\alpha}\right|>\epsilon x\right),
\end{align*}	
where we have used the fact $\Gamma_1^{-1/\alpha}\le M([0,1])$ almost surely, and hence
\begin{equation*}
\liminf_{x\to\infty}\frac{\mathbb{P}( |S(1)|>x,M([0,1])>x)}{\mathbb{P}( |{Z}_1|\Gamma_1^{-1/\alpha}>x,\Gamma_1^{-1/\alpha}>x)}\ge(1+\epsilon)^{-\alpha}.
\end{equation*}
Letting $\epsilon\to 0$, we have
	\begin{equation*}
	\lim_{x\to\infty}\frac{\mathbb{P}( |S(1)|>x,M([0,1])>x)}{\left\{ \mathbb{E}[|{Z}_1|^\alpha\mathbf{1}_{\{|{Z}_1|\leq 1\}} ]+\mathbb{P}(|{Z}_1|> 1)\right\}x^{-\alpha}}=1.
	\end{equation*}
So returning to	\eqref{eq:tail dep coef}, we have
	\begin{equation*}
	\lim_{x\rightarrow\infty}\bb{P}(|S(1)|>x \ | \ M([0,1])>x)= \mathbb{E}[|{Z}_1|^\alpha\mathbf{1}_{\{|{Z}_1|\leq 1\}} ]+\mathbb{P}(|{Z}_1|> 1)\neq 0.
	\end{equation*}
\end{proof}

\subsection{Joint convergence on subordinators}

Recall the right-continuous inverse and the closed range   of a monotonic function is defined as in \eqref{eqlocal} and \eqref{eqrange} respectively. We need the following lemma:
\begin{lemma}\label{lemsub_det}
Suppose $ f,f_n : [0,\infty)\mapsto [0,\infty) $,  $n\in\mathbb{N}$ are non-decreasing unbounded   right-continuous   functions  with right-continuous inverses $f^{\rightarrow}$, $f_n^{\rightarrow}$ and closed ranges $F$ and $F_n$ respectively.   Suppose in addition that $f$ is  strictly increasing,   and the local uniform  convergence holds: for all $t\geq 0$, we have
	\begin{equation}
	\sup_{0\leq s\leq t}\big\lvert f_n(s)- f(s)\big\rvert\rightarrow 0 ~~\text{as }n\to\infty\label{propsub3_det}
	\end{equation}
 Then for all $y\in [0,\infty)$, we have
	\begin{equation}
	\sup_{0\leq x\leq y}\big\lvert f_n^{\rightarrow}(x)-f^{\rightarrow}(x)\big\rvert\rightarrow 0~~\text{as }n\to\infty\label{propsub1_det}
	\end{equation}
and for any sequence $x_n\rightarrow x\in \bb{R}$,
	\begin{equation}
 \left|\rho(x_n,F_n)-\rho(x,F)\right| \rightarrow 0~~\text{as }n\to\infty,\label{propsub2_det}
	\end{equation}
	  where $\rho(x,A) := \inf\{\lvert x -u\rvert : u\in A\}$ for a nonempty subset $A$   of $\bb{R}$.
\end{lemma}

\begin{proof}

    Since $f$ is strictly increasing, $f^{\rightarrow}$ is continuous (see, e.g., \cite[Lemma 13.6.5]{whitt2002stochastic}).   The local uniform convergence \eqref{propsub1_det}  then   follows from  \cite[Corollary 13.6.4]{whitt2002stochastic}.   
 
Next we show \eqref{propsub2_det}.   Due to  the fact  $ |\rho(x_n,F_n)- \rho(x,F_n)|\le |x_n-x|$, it suffices to  prove the special case where $x_n\equiv x$.   { Next, note  that for any $x\in \bb{R}$, 
 since $f$ is non-decreasing and unbounded, we have $f(t)>x$ for some $t>0$ large enough. Then in view of the monotonicity of $f$, for such $t$  we have \[\rho(x,F)=\inf_{0\le s\le t}|x-f(s)|.\] 
On the other hand, since $f_n(t)\rightarrow f(t)$ as $n\rightarrow\infty$ implied by \eqref{propsub3_det}, we have $f_n(t)>x$  and hence
\[\rho(x,F_n)=\inf_{0\le s\le t}|x-f_n(s)|\]
 for all  $n$ large enough.   
Therefore   by   triangular inequalities, we can deduce the following relation   for the aforementioned $t>0$ and all $n$ large enough:
\begin{align*}
\rho(x,F)-\sup_{0\leq s\leq t}\left|f(s)-f_n(s)\right| \le \rho(x,F_n)   
\le 
\rho(x,F)+\sup_{0\leq s\leq t}\left|f(s)-f_n(s)\right|
\end{align*}
The proof is concluded by taking $n\rightarrow\infty$ and applying \eqref{propsub2_det} in the relation above.
 }

\end{proof}

\begin{proof}[Proof of Proposition~{\upshape\ref{thmsub}}] 
The key of the proof is following   coupling result from	  Theorem 15.17 in \cite{kallenberg:2002:foundations}: under the marginal convergence  $\sigma_n(1)\stackrel{d}{\to}\sigma(1)$ as $n\rightarrow\infty$, there exist   subordinators $\widetilde{\sigma}_n\stackrel{d}{=}\sigma_n$  such that  all $t\geq 0$,
\[
\Delta_{n}(t):=\sup_{0\le s\leq t}\left|\widetilde{\sigma}_n(s)-\sigma(s)\right|\stackrel{P}{\to}0
\] 
By     Skorokhod's representation, on a possibly extented    probability space, we can ensure $\widetilde{V}_n\stackrel{a.s.}{\to}\wt{V}$ for some random variables $\widetilde{V}_n\stackrel{d}{=}V_n$, $\wt{V} \stackrel{d}{=}V$, such that $\wt{V}_n$ and $\wt{V}$ are independent of $\wt{\sigma}_n$ and $\sigma$.  Now fix an arbitrary subsequence $S\subset \bb{N}$.  We claim that there exists further subsequence   $S'\subset S$  so that
\[
\bb{P}\left(\lim_{S'{\ni} n\rightarrow\infty}\Delta_{n}(t)=0  \text{ for all }t>0 \right) =1.
\]
Indeed, by the sub-sub-sequence property of convergence in probability, there exist  subsequences $S_k$, $k\in \bb{N}$,  satisfying $S\supset S_1\supset S_{2}\supset\ldots $ and $\lim_{S_k\ni n\rightarrow{\infty}}\Delta_{n}(k)=0$ almost surely for each $k\in \bb{N}$. Then desirable $S'$ can be formed by selecting one element from each $S_k$ in view of the monotonicity of $\Delta_n(t)$ in $t$.

Let $S':=\{n_k\}_{k\in\mathbb{N}}$. Let $\widetilde{L}_{n_k}$ be the right-continuous inverse (i.e., local time) of $\widetilde{\sigma}_{n_k}$ as in  \eqref{eqlocal}. Note that for each $y\geq 0$.
\begin{align*}
&\sup_{0\leq x\leq y}\left|\widetilde{L}_{n_k}((x-\widetilde{V}_{n_k})_+)-L((x-\wt{V})_+)\right|\\
\leq&\sup_{0\leq x\leq y}\left|\widetilde{L}_{n_k}((x-\widetilde{V}_{n_k})_+)-L((x-\widetilde{V}_{n_k})_+)\right|+\sup_{0\leq x\leq y}\left| L((x-\widetilde{V}_{n_k})_+)-L((x-\wt{V})_+)\right|\\
\leq& \sup_{0\leq x\leq y}\left|\widetilde{L}_{n_k}(x)-L(x)\right|+\sup_{0\leq x\leq y}\left| L((x-\widetilde{V}_{n_k})_+)-L((x-\wt{V})_+)\right|.
\end{align*}
So applying \eqref{propsub1_det} of Lemma \ref{lemsub_det}  and the uniform continuity of $L$  on $[0,y]$ to the above, we have
\begin{equation*}
\bb{P}\left(\lim_{k {\to\infty}}\sup_{0\leq x\leq y}\left|\widetilde{L}_{n_k}((x-\widetilde{V}_{n_k})_+)-L((x-\wt{V})_+)\right|=0\right)=1.  
\end{equation*}

Applying the conclusion \eqref{propsub2_det} in Lemma \ref{lemsub_det}, we have 
\begin{equation*}
\bb{P}\left( \lim_{k{\to\infty}} \rho(x,\widetilde{V}_{n_k}+\widetilde{R}_{n_k}) = \rho(x,\wt{V}+R)  
  \text{ for all }x\in [0,\infty) \right)=1,
\end{equation*}
 where $\widetilde{R}_{n_k}$ and $R$ denote the closed  range of $\widetilde{\sigma}_{n_k}$ and $\sigma$, respectively. 
 Since the convergence inside the probability expression above characterizes convergence under the Fell topology   (see, e.g., Theorem 2.2(iii) in \cite{salinetti1981convergence}), we have $\widetilde{V}_{n_k}+\widetilde{R}_{n_k}\rightarrow \wt{V}+R$ almost surely in  $\mathfrak{F}([0,\infty))$. 
The proof is then completed since we have shown a convergence-in-probability version of the conclusion based on the coupling.

\end{proof}

\section*{Declarations}

\subsection*{Availability of supporting data}
Data sharing not applicable to this article as no datasets were generated or analysed during the current study.
\subsection*{Competing interests}
The authors have no relevant financial or non-financial interests to disclose.
\subsection*{Funding}
 No funding was received for conducting this study

\subsection*{Authors' contributions}
The authors contributed equally to this work.

\printbibliography

\end{document}